\documentclass{amsart}
\title[Rational double Hurwitz cycles]{Polynomiality, Wall Crossings and Tropical
Geometry of Rational Double Hurwitz Cycles}
\author{Aaron Bertram, Renzo Cavalieri, Hannah Markwig}
\address{Aaron Bertram - Department of Mathematics, University of Utah, Salt Lake City, UT, USA}
\email{bertram@math.utah.edu}
\address{Renzo Cavalieri - Department of Mathematics, Colorado State University, Fort Collins, CO, USA}
\email{renzo@math.colostate.edu}
\address{Hannah Markwig -  Universit\"at des Saarlandes, Fachrichtung
  Mathematik, Postfach 151150, 66041 Saarbr\"ucken, Germany }
\email{hannah@math.uni-sb.de}

\usepackage{amsrefs}
\usepackage{amsfonts}
\usepackage{amssymb}
\usepackage{amscd}
\usepackage[labelformat=empty]{subfig}
\usepackage{booktabs}
\usepackage{xypic}
\usepackage[dvips]{graphicx,epsfig}
\usepackage{color}

\allowdisplaybreaks[3]

\newcommand{\proj}{{\mathbb{P}^1}}

\newcommand{\bbZ}{\mathbb{Z}}

\newcommand{\bbR}{\mathbb{R}}
\newcommand{\bbC}{\mathbb{C}}
\newcommand{\bbP}{\mathbb{P}}

\newcommand{\bbH}{\mathbb{H}}
\newcommand{\calM}{\mathcal{M}}

\def\bary{\begin{array}} 
\def\eary{\end{array}} 
\def\ben{\begin{enumerate}} 
\def\een{\end{enumerate}}
\def\bit{\begin{itemize}} 
\def\eit{\end{itemize}}

\newcommand{\cT}{\mathcal{T}}

\newcommand{\HH}{\mathcal{H}}

\newcommand {\dunion}{\,\mbox {\raisebox{0.25ex}{$\cdot$} \kern-1.83ex $\cup$}
  \,}

\DeclareMathOperator {\ev}{ev}
\DeclareMathOperator {\ft}{f\/t}
\DeclareMathOperator{\TP}{\mathbb{T}\mathbb{P}^1}
\DeclareMathOperator{\trop}{trop}
\DeclareMathOperator{\fold}{fold}

\def\beq{\begin{equation}}                     %
\def\eeq{\end{equation}}                       %
\def\bea{\begin{eqnarray}}                     
\def\eea{\end{eqnarray}}

\def\IR{{\mathbb R}}

\theoremstyle{plain}

\newtheorem{thm}{Theorem}[section]
\newtheorem{lemma}[thm]{Lemma}
\newtheorem{prop}[thm]{Proposition}

\newtheorem*{conj*}{Conjecture}
\newtheorem{cor}[thm]{Corollary}
\newtheorem*{cor*}{Corollary}
\newtheorem*{co}{Tropical-Classical Correspondence}

\theoremstyle{definition}

\newtheorem{rem}[thm]{Remark}
\newtheorem{defn}{Definition}
\newtheorem{example}{Example}[section]

\newtheorem*{note}{Notation}
\newtheorem{construction}{Construction}
\newtheorem{conv}{Convention}

\DeclareMathOperator{\Aut}{Aut}
\DeclareMathOperator{\val}{val}
\DeclareMathOperator{\triv}{triv}
\DeclareMathOperator{\Cut}{Cut}
\DeclareMathOperator{\Glue}{Glue}
\DeclareMathOperator{\br}{br}
\DeclareMathOperator{\st}{st}
\DeclareMathOperator{\gl}{gl}

\newcommand{\GIT}[1]{/\!\!/_{\kern-.2em #1 \kern0.1em}}

\newcommand{\fc}{\mathfrak{c}}

\newcommand{\mon}{\overline{M}_{0,n}}
\newcommand{\LM}{\overline{M}_{0,2+r}(1^2,\varepsilon^{r})}
\newcommand{\Htk}{\tilde\bbH_k(\mathbf{x})}

\numberwithin{equation}{section}

\begin{document}

\begin{abstract}
We study rational double Hurwitz cycles, i.e. loci of marked rational stable curves admitting a map to the projective line  with assigned ramification profiles over two fixed branch points. Generalizing the phenomenon observed for double Hurwitz numbers, such cycles are piecewise polynomial in the entries of the special ramification; the chambers of polynomiality and wall crossings have an explicit and ``modular'' description. A main goal of this paper is to simultaneously carry out this investigation for the corresponding objects in tropical geometry, underlining a precise combinatorial duality between classical and tropical Hurwitz theory.  
\end{abstract}

\maketitle
\tableofcontents


\section{Introduction}

Hurwitz theory is the study of maps of algebraic curves, viewed as ramified covers of orientable surfaces. At the intersection of geometry, representation theory and combinatorics, it is an area that naturally lends itself to making bridges and connections. In this paper we study the combinatorial structure of certain Hurwitz spaces via a parallel investigation in tropical geometry.

\subsection{Summary}

The study of the connections between classical and tropical Hurwitz theory was initiated in \cite{CJM10} and \cite{CJMwc};  the tropical point of view provided a combinatorial interpretation of double Hurwitz numbers which was very well tuned to describing the polynomial aspects and the wall crossing phenomena occurring in the theory. We  continue this exploration by studying rational double Hurwitz loci inside spaces of (relative/tropical) stable maps, generically parameterizing covers of $\bbP^1$ with two prescribed ramification profiles and a part of the branch divisor fixed, and their pushforwards to the moduli space of curves, which we call double Hurwitz cycles. Besides the genus, the discrete invariants we fix are the total length $n$ of the special ramification profile, and the dimension of the loci we want to study. Then we study families of Hurwitz loci parameterized by integral points in an $(n-1)$ dimensional vector space. 

On the classical side,  we realize Hurwitz loci as the pullback via a natural branch morphism of appropriate strata in spaces of pointed  chains of projective lines (Losev-Manin spaces, Section \ref{sec:mscm}). This gives a boundary expression for Hurwitz cycles where the coefficients are piecewise polynomials in the entries of the special ramification data (Theorem \ref{thm:poly}).  

On the tropical side, we realize the Hurwitz loci as tropical Gromov-Witten cycles. Our tropical Hurwitz cycles are balanced polyhedral complexes; their topology is constant in the chambers of polynomiality of (classical) Hurwitz cycles, whereas their geometry (affine integral structure, weights and coordinates of vertices) varies in a polynomial way in terms of the special ramification profiles. 

Naturally, we also study the connection between classical and tropical Hurwitz cycles (Section \ref{sec:tcc}) and observe a natural combinatorial duality between tropical and classical strata. To be more precise, the stratification on the tropical side is the polyhedral structure inherited from the moduli space of tropical curves. The stratification on the classical side is the usual stratification in boundary classes. For Hurwitz cycles of dimension $d$, $k$-dimensional classical strata correspond to $d-k$-dimensional tropical strata, and the combinatorial type of the tropical stratum can be encoded in terms of the dual graph of the classical stratum.

We conclude the paper by studying how Hurwitz cycles vary across the walls of the chambers of polynomiality, both on the classical and on the tropical side. In a similar fashion to Hurwitz numbers, the wall crossing  formulae have an inductive  form:  the cycles in the formula can be obtained as pushforwards via appropriate gluing morphism  of pairs of boundary strata  coming from Hurwitz cycles where the profile data is split according to the equation of the wall, and the dimensions are split in all possible ways adding up to the correct one. 
Even though the final form of the tropical and classical wall-crossing formulae is essentially the same, there are some subtleties involved in even making sense of what a tropical wall crossing formula may be: that's why we treat the two cases separately. The classical wall crossing formula is Theorem \ref{thm:wc}, the (cleanest form of  the) tropical one Corollary  \ref{cor:twc}.

To make our exposition easier to follow, throughout the paper we use the one-dimensional case ((tropical) Hurwitz curves) as a running example.

\subsection{Context, Motivation and further directions of our research}

Because of the many and diverse categories equivalent to curves and their maps,  Hurwitz theory is by nature ``interdisciplinary'': exploring the dictionary between the tropical and classical approaches to Hurwitz theory is a natural thing to do. It has already been fruitful and hopefully will bear even more applications in the future. 

Before Hurwitz theory made its appearance in tropical geometry, the area of tropical enumerative geometry was pioneered by Mikhalkin's celebrated Correspondence Theorem which relates classical numbers of plane curves to their tropical counterparts \cite{Mi03}. Nowadays, numbers of tropical curves can (at least in the rational case) be understood as intersection products in an appropriate moduli space of tropical curves, just as in the classical world. For higher genus, understanding the appropriate moduli space of tropical curves and its connection to the moduli space of algebraic curves is an active area of research (see e.g.\ \cites{CMV12,Cap12a}). Also in the rational case, the connection between the intersection theory of the moduli space of algebraic curves and the moduli space of tropical curves is not yet completely understood. Combinatorial dualities such as the one we observe in Section \ref{sec:tcc} are present in many situations, but in our opinion they do not fully explain the success of tropical methods in enumerative geometry. We expect deeper connections to be discovered. Our paper contributes some interesting and geometrically meaningful families of cycles in the intersection ring of tropical $\mathcal{M}_{0,n}$, and makes important steps in understanding the correspondence of these cycles to their classical counterparts. We hope to extent the study to higher genus, and to contribute to the understanding of the deeper connections between moduli spaces of algebraic and tropical curves.

Classically double Hurwitz loci were a key ingredient in the proof of  Theorem $\star$, the main result of  \cite{gv:rvl}: tautological classes in the moduli space of curves of degree greater than $g-1$ (say $g+k$ with $k$ a non-negative integer) must admit an expression supported on the boundary, and more specifically on strata parameterizing curves with at least $k$ rational components. Then again they were applied to the study of tautological classes in \cite{gjv:last}, even though in neither of these cases they were viewed as families of loci with any sort of polynomial structure. This makes its appearence for the case of $0$-dimensional cycles, or more mundanely double Hurwitz numbers, in \cite{gjv:ttgodhn}. After \cites{ssv:cbodhn,CJM10,CJMwc} the algebro-combinatorial aspects of the piecewise polynomiality
of double Hurwitz numbers are well understood, leading the way to some really interesting geometric questions: 
\begin{description}
\item[ELSV-type formula for double Hurwitz numbers] can double Hurwitz numbers be obtained as intersections of tautological classes on some family of birational moduli spaces --- constant in the polynomiality chambers ---  in a way that naturally explains the structure of the polynomials? Can the wall crossings be somehow seen as coming from the birational transformations occurring when crossing the walls?
\item[Higher dimensional Hurwitz loci] How well does the piecewise polynomial structure carry over to higher dimensional loci? In particular can we understand full Hurwitz spaces (or rather their compactifications such as Admissible Covers or Relative Stable Maps) as tautological classes in the moduli space of curves?
\end{description}
This paper provides an exhaustive answer for the second question in genus $0$: here the full moduli space is birational to $\overline{M}_{0,n}$, and we show that every time we  increase the codimension by one by fixing a simple branch point in the branch divisor we obtain a polynomial class of one degree higher. In genus $0$ an ELSV formula  is trivially showed to hold in the one part double Hurwitz number case (\cite{gjv:ttgodhn}), and we are hopeful that a geometric understanding of the wall crossings will be complete soon. In higher genus, the situation is both more complicated and more interesting: here the Hurwitz moduli space represents a codimension $g$ tautological class, which has recently been at the center of attention because of its connections with symplectic field theory (\cite{eliashberg}).  Recently Hain (\cite{H11}) has produced  tautological  classes on $\overline{M}_{g,n}$,  which agree with double Hurwitz loci when restricted to the partial compactification of curves of compact type (however simple intersection computations show that Hain's class does not agree with either the Admissible Cover nor the Relative Stable Map compactification of the Hurwitz space already in genus one). Interestingly, Hain's class is homogeneous polynomial of degree $2g$. Understanding how Hain's class compares with Admissible Covers or Relative Stable Maps, besides being a very interesting question on its own, is likely a useful ingredient in the quest for an ELSV formula for double Hurwitz numbers. Our work is aiming in that direction in a couple different ways. On the one hand we interpolate a family of classes  starting from the zero dimensional loci which we understand very well. On the other hand we make a connection with tropical geometry, which is usually well tuned to give information about the  deeper boundary strata of the classical moduli spaces of curves.

\subsection{Acknowledgements}
This work is the result of a two week long \textit{Research in Pairs} at the Oberwolfach Institute for Mathematics, which the authors thank for its hospitality. The second author was at the intersection of supports by a Simons Collaboration Grant and NSF grant  DMS110549. The third author was partially supported by DFG-grant MA 4797/1-2.
 


\section{Background and Notation}
In this section we recall the basic definitions and constructions that are needed for the set-up of the theory. 
\subsection{Moduli Spaces of Curves and  Maps}
\label{sec:mscm}

We assume that the reader is familiar with $\mon$, the moduli space of rational pointed stable curves, a smooth projective variety of dimension $n-3$. The Chow ring of $\mon$ is generated by irreducible boundary divisors, with the only relations (besides the obvious ones given by empty intersections) generated by the WDVV relations (\cite{sk:m0n}). Irreducible boundary strata are identified by their dual graph: given a graph $\Gamma$, we denote the corresponding stratum by $\Delta_\Gamma$.

In \textit{weighted stable curves} one tweaks the stability of a rational pointed curve $(X= \cup_j X_j, p_1, \ldots, p_n)$ by assigning weights $a_i$ to the marked points and requiring the restriction to each $X_j$ of  $\omega_X +\sum a_ip_i$ to be ample (this amounts to the combinatorial condition that  $\sum_{p_i\in X_j} a_i + n_j>2$, where $n_j$ is the number of shadows of nodes on the $j$-th component of the normalization of $X$). 

When two points are given weight $1$ and all other points  very small weight, the space $\LM$ is classically known as the \textit{Losev-Manin} space {\cite{lm:lms}}: it parameterizes chains of $\bbP^1$'s with the heavy points on the two terminal components and light points (possibly overlapping amongst themselves) in the smooth locus of the chain.

Let  $\mathbf{x}$ be an $n$-tuple  of integers adding to $0$, and denote $\mathbf{x^+}$ (resp. $\mathbf{x^-}$) the sub-tuple of positive (resp. negative) parts. We consider moduli space of relative stable maps $\overline{M}_0(\bbP^1; \mathbf{x^+}0, \mathbf{x^-}\infty)$ and their ``rubber'' variant  $\overline{M}^\sim_0(\bbP^1; \mathbf{x^+}0, \mathbf{x^-}\infty)$ (see \cites{gv:rvl,mp:tvgwt}).  An important technical detail is we mark the preimages of the relative points. In order to mark some of the simple ramification points on the source curve, we introduce a space of relative weighted stable maps, where in addition to $0$ and $\infty$ there are $j$ moving simple transposition points that are marked and given weight $\varepsilon$.  These spaces  are typically denoted $\overline{M}^\sim_0(\bbP^1; \mathbf{x^+}0, \mathbf{x^-}\infty, \varepsilon t_1, \ldots, \varepsilon t_j)$. 
\begin{note}
In all our spaces of maps we make the notation lighter by  forgetting the target (always $\bbP^1$),  noting that $\mathbf{x}$ gives sufficient information to   determine the  relative points fixed at $0 $ and $\infty$ and that the additional transposition points are understood to be ``light''. For example we write $\overline{M}^\sim_0(\mathbf{x}, t_1, \ldots, t_j)$ for $\overline{M}^\sim_0(\bbP^1; \mathbf{x^+}0, \mathbf{x^-}\infty, \varepsilon t_1, \ldots, \varepsilon t_j)$. 
\end{note}

There is a natural stabilization map $\st$ to $\mon$ that forgets the map and remembers the (marked) points over $0$ and $\infty$, and a branch map to an appropriate quotient of a Losev-Manin space, recording the position of the $r+2=n$ branch points. 
Since a degree $0$ divisor on $\bbP^1$ determines a rational function up to a multiplicative constant, the map $\st: \overline{M}^\sim_0(\mathbf{x}) \to \mon$ is birational. Marking $j$ simple transpositions makes $\st$ into a degree ${{r}\choose{j}}$ cover. The branch map is a cover of the Losev-Manin space of degree the double Hurwitz number $H_0(\mathbf{x})$. 

\subsubsection{Multiplicities of boundary strata.}
Boundary strata in moduli spaces of relative stable maps corresponding to breaking the target are naturally described in terms of products of other moduli spaces of relative stable maps. It is important to keep careful track of various multiplicities coming both from combinatorics of the gluing and infinitesimal automorphisms (see \cite{gv:rvl}*{Theorem 4.5}). 
Let $S$ be a boundary stratum in $\overline{M}^\sim_0(\mathbf{x})$, parameterizing maps to a chain $T^N$  of $N$ projective lines. $S$ can be seen as the image of:
$$
\gl: \prod_{i=1}^N \mathcal{M}^\bullet_i \to S \subset \overline{M}^\sim_0(\mathbf{x}),
$$

where the $\mathcal{M}^\bullet_i$ are moduli spaces of possibly disconnected relative stable maps, where the relative condition imposed at the point $\infty$ in the $i$-th line matches the condition at $0$ in the $(i+1)$-th line. We denote by $\mathbf{z_i}=(z_i^1, \ldots, z_i^{r_i})$ such relative condition and by abuse of notation we say it is the relative condition at the $i$-th node of $T^N$. Then,
\begin{equation}
\label{bound:multi}
[S] = \prod_{i=1}^{N-1}\frac{\prod_{j=1}^{r_i} z_i^j}{|\Aut(\mathbf{z_i})|} \left[\gl_\ast\left( \prod_{i=1}^N \mathcal{M}^\bullet_i \right)\right].
\end{equation}

Equation \eqref{bound:multi} seems horrendous, but it amounts to the following recipe: the general element in $S$ is represented by a map from a nodal curve  $X$ to $T^N$, with matching ramification on each side of each node of $X$. The multiplicity $m(S)$ is the product of ramification orders for each node of $X$ divided by the (product of the order of the group of) automorphisms of each partition of the degree prescribing the ramification profile over each node of $T^N$.


\subsection{Tropical Geometry}
\label{sec:tg}
We assume that the reader is familiar with tropical varieties and tropical cycles in a vector space with a lattice, i.e.\ with weighted polyhedral complexes (possibly with negative weights in the case of cycles) satisfying the balancing condition around each cell of codimension one.
The exact polyhedral complex structure is not important. We do not distinguish between equivalent tropical cycles, i.e.\ cycles which allow a common refinement respecting the weights. 

 A rational function
  on a tropical cycle is a continuous function that is
  affine on each cell, and whose linear part is
  integer. To a tropical cycle $X$ and a rational function $\varphi$, we can associate the divisor $\varphi\cdot X$, a tropical subcycle of codimension one supported on
  the subset of $X$ where $\varphi$ is not linear \cite{AR07}*{Construction 3.3}.
We can also form multiple intersection products $\varphi_1 \cdot \ldots \cdot
  \varphi_m \cdot X$. They are commutative by \cite{AR07}*{Proposition 3.7}. The weights of cells of  intersection products can be computed locally as follows.
  
\begin{rem}\label{rem-intersect}
  Let $h_1, \ldots, h_m$ be linearly independent integer linear
  functions on $\mathbb{R}^n$. By $H : \mathbb{Z}^n \rightarrow \mathbb{Z}^m $ we
  denote the linear map given by $x \mapsto (h_1(x), \ldots, h_m(x))$. Consider
  the rational functions $ \varphi_i = \max\{h_i,p_i\}$ on $\mathbb{R}^n$, where $p_i$ are fixed constants in $\mathbb{R}$. These rational functions give rise
  to an intersection product, which obviously consists of only one cone with 
 weight equal to the order of the torsion part of $ \mathbb{Z}^m / \text{Im}(
  H)$, i.e. the greatest common divisor of
  the absolute values of the maximal minors of $H$ (see e.g.\ \cite{MR08}*{Lemma 5.1}). 
  \end{rem}

A morphism between tropical cycles
is a locally affine linear map, with
the linear part induced by a map between the underlying lattices.
 A rational function $\varphi$ on a tropical cycle $Y$ can be pulled back along a morphism $f:X\rightarrow Y$
  to the rational function $f^*(\varphi) = \varphi \circ f$ on $X$. Also, we can push forward
  subcycles $Z$ of $X$ to subcycles $f_*(Z)$ of $Y$ \cite{AR07}*{Proposition
  4.6 and Corollary 7.4}. 
\vspace{0.3cm}

We refer the  reader  to \cite{GKM07,Mi07,CJM10} for comprehensive background on moduli spaces of tropical curves and maps. A (marked, rational, abstract) tropical curve is a metric tree $
\Gamma $ without 2-valent vertices. Edges leading to 1-valent vertices have infinite length, are marked by the numbers $1,\ldots,N$ and are called ends. The space of all marked tropical curves with
$N$ ends is denoted $ \calM_{0,N} $. It can be embedded into $\mathbb{R}^{\binom{N}{2}-N}$ using the distance map. It follows from \cite{SS04a}*{Theorem 3.4},
\cite{Mi07}*{Section 2}, or \cite{GKM07}*{Theorem 3.7} that $\calM_{0,N}$ is a
tropical variety which is even a fan. All top-dimensional cones have weight one. 
For a subset $I\subset\left[N\right]$ of cardinality $1<|I|<N-1$ define $v_I$
to be the image under the distance map of a tree $\Gamma$ with exactly one bounded
edge of length one, the marked ends $i\in I$ on one side and the
ends $i$ for $i\notin I$ on the other. By abuse of notation, we often do not distinguish between a tree and its image under the distance map. We consider $\mathbb{R}^{\binom{N}{2}-N}$ not with its usual lattice, but with the lattice generated by the $v_I$.
The set of curves of a given combinatorial type is the interior of a
$k$-dimensional cone of $\calM_{0,N}$, where $k$ denotes the number of bounded
edges. The vectors $v_I$ generate the rays of $\calM_{0,N}$.

We denote by $\TP= \mathbb{R}\cup {\pm \infty}$ the simplest model of tropical $\mathbb{P}^1$.
A tropical cover is a map $h:\Gamma \rightarrow \TP$ where $\Gamma$ is a tropical curve and $h$ is a continuous map which is integral affine on each edge. For each bounded edge, the stretching factor with which it is mapped to $\mathbb{R}\subset\TP$ is called its weight.  
Edges which are contracted have weight $0$. Two tropical covers are called isomorphic if they differ by a translation of the base $\TP$.

We sometimes fix a reference orientation for the edges of $\Gamma$. We then use the convention that (nonzero) weights are positive whenever the image of the tail in $\mathbb{R}\cup {\pm \infty}$ is smaller than the image of the head, and negative otherwise. At each vertex, a tropical cover satisfies the balancing condition, i.e.\ the sum of the weights of incoming edges equals the sum of the weights of outgoing edges.
The degree $\mathbf{x}$ of a tropical cover is the multiset of weights of its non-contracted ends, where we make the convention that all ends point away from their 1-valent vertex.

We fix  a degree $\mathbf{x}$ and set $n=\ell(\mathbf{x})$. The space of isomorphism classes of tropical covers of degree $\mathbf{x}$ and with additional $r$ marked ends which are contracted to points is called $ \calM_{0,r}(\TP,\mathbf{x}) $.

\begin{conv}
\label{conv}
For practical purposes in computations, it is useful to pick a distinguished representative for each equivalence class of tropical covers. By
convention, we introduce a contracted end labeled  $n+1$ and require that 
it is mapped to $0\in \TP$ by the map $h$. We say the end $n+1$ fixes a parameterization of $\TP$.
\end{conv}

By \cite{GKM07}*{Proposition 4.7}, $ \calM_{0,r}(\TP,\mathbf{x}) $ is a
  tropical variety. It can be identified with $\calM_{0,N}$, where $N=n+r$, via the
  map that forgets the map $h$.
If we later want to count tropical covers using the space 
$ \calM_{0,r}(\TP,\mathbf{x}) $ we overcount in the situation where we have multiple ends of the same weight, since we mark all the ends here and thus make them artificially distinguishable even if they have the same weight. For counting purposes, we thus have to divide by $|\Aut(\mathbf{x})|$. Since we are mostly interested in the situation where the degree $\mathbf{x}$ is variable, this does not play an important role.

To compute intersection products in cells of
  tropical moduli spaces, we work in the local coordinates of the appropriate cone of $\calM_{0,N}$, i.e.\ in coordinates given by the lengths of each bounded edge.

For a subset $I\subset\left[N\right]$, there
is a forgetful map $\ft_I:\calM_{0,N}\longrightarrow \calM_{0,N-|I|}$ which
maps a tropical curve to the curve where we remove all ends with labels in $I$ (and possibly straighten $2$-valent vertices). Forgetful maps are
morphisms by \cite{GKM07}*{Proposition 3.9}.
The forgetful map $\ft:  \calM_{0,r}(\TP,\mathbf{x}) = \calM_{0,N} \rightarrow \calM_{0,n}$ which forgets all contracted ends is most important.

 For each contracted end with label $i$  the evaluation map
    $ \ev_i : \calM_{0,r}(\TP,\mathbf{x}) \rightarrow \mathbb{R}\subset\TP $
  assigns to a tropical curve the point in $\mathbb{R}\subset \TP$ to which the end is contracted by $h$. It is shown in \cite{GKM07}*{Proposition
  4.8} that these maps are morphisms.

For each marked end $i$  the tropical Psi-class $\Psi_i$, as a divisor, consists of all closed codimension one cones of $\calM_{0,N}$ parameterizing graphs  where the end $i$ is adjacent to a 4-valent vertex \cite{Mi07}. All these cones come with weight one. It is also known how to intersect tropical Psi-classes \cite{KM07}: every top-dimensional cone appearing in the intersection $\Psi_1^{k_1}\cdot \ldots \cdot \Psi_N^{k_N}$ corresponds to a combinatorial type satisfying the following: if the ends $i_1,\ldots,i_s$ are adjacent to a vertex $V$ then this vertex has valency $k_{i_1}+\cdots+k_{i_s}+3$. We associate the weight $(k_{i_1}+\cdots+k_{i_s})! \cdot (k_{i_1}!\ldots k_{i_s}!)^{-1}$ to the vertex $V$. The weight of a cone equals the product of these vertex weights for the corresponding combinatorial type.

For a subset $I\subset \{1,\ldots,N\}$ with $N-2 \geq |I|\geq 2$ we define the tropical boundary divisor $D_{I}$ to be the divisor cut out by the rational function $\varphi_{I}$ that sends $v_{I}$ to one and all other vectors $v_{I'}\neq v_I$ to zero \cite{Rau08}*{Definition 2.4}.

\subsection{Combinatorial Set-Up}
\label{sec:csu}

In this section we introduce some basic combinatorial objects that play a role in the development of Hurwitz cycles. We think of Hurwitz loci as functions of the special ramification profiles. We define the domain of definition of such functions and outline the linear structure it has. 

\begin{defn}
Fix a positive integer $n$. The lattice
$$
\mathcal{H}=\bbZ^n \cap \left\{ \sum_{i=1}^n x_i = 0 \right\}
$$
is the parameter space for double Hurwitz cycles. 
For any proper subset $I\subset \{1, \ldots, n\}$, the hyperplane
$$
W_I= \left\{ \sum_{i\in I} x_i = 0 \right\}
$$
is called a \textbf{wall}.  Any connected component of $\mathcal{H}\otimes \bbR \smallsetminus \cup_I W_I$ is called a \textbf{chamber}  (of polynomiality) and denoted by $\fc$.
\end{defn}

Next we introduce some notation for sets of decorated trees.
\begin{defn}
Let $\mathbf{x}\in \HH$. We denote by $\cT_{r-k}(\mathbf{x})$ the set of (non-metric) trees with $n$ ends labelled by the $x_i$'s and $r-k$ internal vertices of valence at least $3$.
We denote by $\cT^{\fc}_{r-k}(\mathbf{x})$ the set of tropical covers whose underlying tree is in $\cT_{r-k}(\mathbf{x})$. Combinatorially, trees in $\cT^{\fc}_{r-k}(\mathbf{x})$ are further decorated with the following data:
\begin{enumerate}
	\item A direction for the edges.
	\item Positive weights for each edge in such a way that the balancing condition is satisfied at every vertex.
	\item A total ordering of the vertices compatible with the direction of the edges.
\end{enumerate}
\end{defn}

Decorations (1) and (2) can in fact be deduced uniquely from imposing the balancing condition, and depend on the chamber $\fc$ that $\mathbf{x}$ belongs to.

\begin{lemma}
\label{lem:flat}
Any $\Gamma$ in $\cT_{r-k}(\mathbf{x})$ can be  ``lifted'' to a tropical cover as in Section \ref{sec:tg}.
\end{lemma}
\begin{proof}
For most combinatorially versed people Lemma \ref{lem:flat} is obvious. We include a brief sketch for the combinatorially impaired such as the majority of the authors of this paper (see \cite{GKM07}*{Lemma 4.6} for a more honest discussion).
Any tree with leaves labelled by an $n$-tuple of integers adding to $0$ can be given positive weights and orientations to all internal edges in a unique way by imposing the balancing condition:  give an arbitrary orientation to all edges in the tree. Cut any internal edge to obtain two smaller connected graphs. The cut edge can be assigned weight by requesting that each the sum of labels on each of the two connected graphs is $0$. Now the statement follows by induction. Wherever  one obtained a negative weight at an internal edge, switch the orientation of the edge with respect to the arbitrary one chosen earlier and switch the sign of the weight. Having given an orientation to any internal edge immediately allows to flatten the graph, concluding the proof. 
\end{proof}

\begin{defn}
\label{mult}
The number of ways that a graph $\Gamma$ can be flattened is  the number of vertex orderings for the graph that are compatible with the direction of the edges of the graph.  We call such number $m_\fc(\Gamma)$. If the chamber $\fc$ is clear from the context, we also write $m(\Gamma)$. \end{defn}

In other words, the forgetful morphism $f:\cT^\fc_{r-k}(\mathbf{x})\to \cT_{r-k}(\mathbf{x})$ is surjective and $m_\fc(\Gamma)= |f^{-1}(\Gamma)|$.

\begin{defn}
\label{phi}
For $\Gamma \in \cT^\fc_{r-k}(\mathbf{x}) $ (or in $\cT_{r-k}(\mathbf{x})$ ) we define $\varphi(\Gamma)$ to be the product of the weights of all internal edges of $\Gamma$.
\end{defn}

Because in each chamber $\fc$ all edge weights are linear homogeneous polynomials in the $x_i$'s, the functions   $\varphi(\Gamma)$ are homogeneous polynomials in each chamber $\fc$ of degree equal the number of internal edges of $\Gamma$.


\section{Hurwitz Loci and Classes}
\label{sec:hl}

In this section we define Hurwitz classes and study their (piecewise) polynomial properties. We say that a family of Chow cycles $\alpha(\mathbf{x})$  in the moduli space of rational pointed stable curves is polynomial of degree $d$ and dimension $k$ if  $\alpha(\mathbf{x}) \in Z_k(\mon)[x_1, \ldots, x_n]_d$. This is equivalent to 
$\alpha$  having an expression as a combination of dimension $k$ boundary strata with coefficients polynomials in the $x_i$'s of degree $d$. 

\begin{figure}[tb]
\input{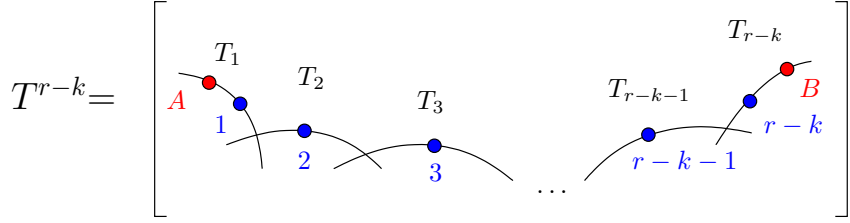}
\caption{The union of strata we denote by $T^{r-k}\subset \LM$ paramereterizes  chains of $r-k$ rational curves, where the $i$-th component hosts the $i$-th (light) marked point, and the remaining $k$ points are distributed in all possible ways among the various twigs.}
\label{fig:T}
\end{figure}

\begin{defn}
Let $\mathbf{x}\in \mathcal{H}$. Consider the union  of boundary strata $$T^{r-k}\subset \LM$$ parameterizing  chains of $r-k$ projective lines, where the $i$-th component hosts the $i$-th (light) marked point, as illustrated in Figure \ref{fig:T}. Referring to Figure \ref{fig:natmorph} for the names of the natural morphisms, we define the $k$-dimensional Hurwitz cycle:
\begin{equation}
\bbH_k(\mathbf{x}):=\st_\ast( \br^\ast(T^{r-k})) \in Z_k(\mon).
\end{equation}
Sometimes we want to look at the $k$-dimensional Hurwitz locus in the appropriate space of maps. We make the definition:
\begin{equation}
\tilde\bbH_k(\mathbf{x}):= \br^{-1}(T^{r-k}) \subseteq \overline{M}^\sim_{0}(\mathbf{x},t_1,\ldots, t_{r-k}).
\end{equation}
\end{defn}

\begin{figure}[tb]
\centerline{
\xymatrix{
\tilde{\bbH}_k(\mathbf{x})   \ar[d]  \ar[r]   &  \overline{M}^\sim_0(\mathbf{x}, t_1, \ldots, t_j)  \ar[d]^{\br}\ar[r]^{\st}& \mon\\
 [T^{r-k}] \ar[r]&  [\LM/S_k] & \\
}}
\caption{The $k$-dimensional Hurwitz locus is the inverse image via the branch map of the stratum $T^{r-k}$.}
\label{fig:natmorph}
\end{figure}

\begin{rem}
Hurwitz loci were introduced in \cite{gv:rvl}*{Section 4.4}. The only difference  in our definition is that we ``mark'' the branch points that we are fixing. The use of the more refined branch morphism to the Losev-Manin space gives us a more convenient expression of the locus in terms of the pull-back of a boundary stratum in $\LM$. The next Lemma establishes the equivalence with Graber-Vakil's definition and is parallel to the definition we make on the tropical side.
\end{rem}

\begin{lemma}
Consider the space of relative stable maps to a parameterized $\bbP^1$, with   the natural stabilization morphism $\st: \overline{M}_{0,r-k}(\mathbf{x})\to \mon$.  Let $\psi_i$ denote the first Chern class of the $i$-th cotangent line bundle, and $\ev_i$ the $i$-the evaluation morphism. Then
\begin{equation}
\bbH_k(\mathbf{x})=\st_\ast\left(\prod_{i=1}^{r-k} \psi_i \ev^\ast_i([pt])\right).
\end{equation}
\end{lemma}
\begin{proof}
Informally, the result is a combination of the three following facts: putting a $\psi$ class is equivalent to requesting the mark to be a point of simple ramification for the map (\cite{op:cc}); fixing one of the branch points gives the equivalence with the space of rubber-maps (\cite{gv:rvl}*{Lemma 4.6}); via degeneration formula, any additional branch point that is fixed amounts to breaking the target into a chain with one more projective line. We expand on this third fact for the benefit of a reader who is not familiar with degeneration techniques.

We assume by induction that the class $\prod_{i=1}^j \ev_i^\ast{[pt]}$ can be represented by the cycle of maps to a chain $T^j=\cup_{i=1}^jT_i$ of $j$ projective lines with the $i$-th mark on the $i$-th line. Consider the trivial family $X=\bbC\times T^j$ and note that it comes with two horizontal sections: $\mathbf{0_1}$ on $T_1$ and $\mathbf{\infty_j}$ on $T_j$. On $T_j$ consider the section $s(t)=1/t$ that meets the section  $\mathbf{\infty_j}$ at $t=0$. Now let $X'$ be the blow up the point $(0,\infty_j)$ in $X$ and consider the locus $\ev_{j+1}^{\ast}([s(t)])$ in the space of relative stable maps to $X'$ of degree $dF$ (where $F$ is the class of a fiber of the family and $d= \sum x_i^+$ ), relative to the divisor $\mathbf{x^+0_1}+\mathbf{x^-\infty_j}$. The degeneration formula (\cites{lr:df, l:df})  asserts that it is equivalent to represent $F$ by a general fiber or by the central fiber of the family. In the first case we have $\prod_{i=1}^{j+1} \ev_i^\ast{[pt]}$, in the second case we obtain the cycle of maps to the chain $T^{j+1}$ with the $i$-th mark on the $i$-th line.  
\end{proof}

The description of Hurwitz loci in terms of boundary strata in spaces of maps, together with an analysis of the multiplicities of the push-forwards to $\mon$ naturally leads to discover the (piecewise) polynomiality of the Hurwitz cycles.

\begin{thm}
\label{thm:poly}
For $\mathbf{x}\in \fc$, $\bbH_k(\mathbf{x})$ is a homogeneous polynomial cycle of degree $n-3-k$. 
\end{thm}
\begin{proof}
Theorem \ref{thm:poly} is a consequence of Lemmas \ref{spec}, \ref{flatten} and \ref{coeff}, where we also give an explicit description of the (polynomial) multiplicity of each boundary stratum.
\end{proof}

A simple observation that gives us a surprising amount of mileage, is that many boundary strata in $\Htk$ don't survive being pushed forward to $\mon$.

\begin{lemma}
\label{spec}
Let  $\mathbf{x}\in \fc$ and $\tilde\Delta$ an irreducible component of $\tilde\bbH_k(\mathbf{x})$, corresponding to a stratum of maps to the rational chain $T=T_1\cup\ldots\cup T_{r-k}$. Then $\pi_\ast(\tilde\Delta)\not=0$ if and only if for every $i=1, \ldots, r-k$ there is exactly one connected rational curve mapping non-trivially to $T_i$. 
\end{lemma}
\begin{proof}
Let $f:X\to T$ correspond to a general point in $\tilde\Delta$.
Over each $T_i$ there must be at least one connected component mapping non-trivially. By a trivial dimension count the push forward of $\tilde\Delta$ does not vanish precisely when the stabilization of $X$ is a rational curve with $r-k$ components and therefore there cannot be more than one non-trivial component over each $T_i$. 
\end{proof}

Next we observe that each boundary stratum of dimension $k$ in $\mon$ appears as the push forward of some stratum in $\Htk$.

\begin{lemma}
\label{flatten}
Let   $\mathbf{x} \in \fc$,  $\Gamma\in \cT_{r-k}(\mathbf{x})$. Then there exists  an irreducible component $\tilde\Delta$ in $\tilde\bbH_k(\mathbf{x})$ such that the stabilization of the source curve of the general element of $\tilde\Delta$ has dual graph $\Gamma$.
\end{lemma}
\begin{proof}
This statement becomes  transparent after noting that boundary strata in $\overline{M}^\sim_{0}(\mathbf{x},t_1,\ldots, t_{r-k})$ are in bijective correspondence with ``tropical dual graphs''. Given a boundary stratum  $[S]$ whose general element is given by a map $f: X\to T^{r-k}$: 
\begin{itemize}
\item the chain $T^{r-k}$ is replaced by the interval $[0, r-k+1]$ with the $i$-th projective line corresponds to the point $i$ and $i$-th node corresponding to the segment $(i, i+1)$; the points $0$ and $r-k+1$ correspond to the relative points.
\item over point $i$ draw one vertex for each connected component of $f^{-1}(T_i)$; 
\item over segment $(i,i+1)$ draw one edge for each node  of $X$ above the $i$-th node of $T^{r-k}$; the edge connects the appropriate vertices and is weighted with the ramification order at the node;
\item over $[0,1)$ and $(r-k,r-k+1]$ draw edges corresponding to the relative condition $\mathbf{x}$.
\end{itemize}
We call the graph thus obtained the \textit{pre-stable} tropical dual of $[S]$.  We stabilize the graph by forgetting all two-valent vertices to obtain what we call the tropical dual of $[S]$. 
This procedure is illustrated in Figure \ref{fig:flatten}  and clearly it can be inverted to identify a boundary stratum of maps given a tropical cover. With this translation Lemma \ref{flatten} is equivalent to the purely combinatorial statement of Lemma \ref{lem:flat}.
\end{proof}


\begin{figure}[tb]
\input{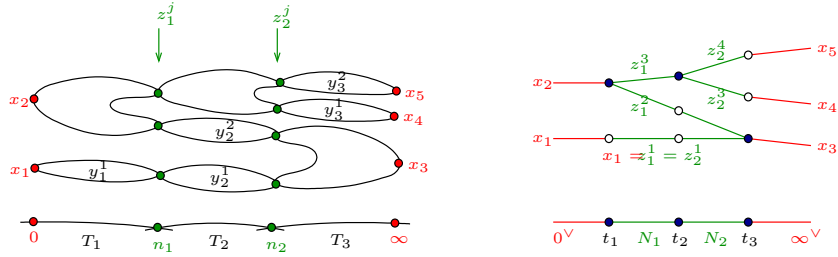}

\caption{The construction of the tropical dual graph associated to a boundary stratum of relative maps. The $y_i^j$ are the degrees of the trivial covers of the $i$-th twig, and the $z_i^j$ the ramification orders over the $i$-th node. Note that the $z$'s on either side of a trivial cover are equal to the corresponding $y$ (e.g $x_1=y_1^1=z_1^1=y_2^1=z_2^1$). This makes it possible to erase the two valent vertices and have the $z_i^j$ become the weights of the edges of the tropical dual graph.}
\label{fig:flatten}
\end{figure}

\begin{lemma}
\label{coeff}
Let   $\mathbf{x} \in \fc$,  $\Gamma\in \cT_{r-k}(\mathbf{x})$ and $\Delta_{\Gamma}$ the corresponding boundary stratum in $\overline{M}_{0,n}$.  Let $m_\fc(\Gamma)$  and $\varphi(\Gamma)$ as in Definitions \ref{mult}, \ref{phi} and $f:\cT^\fc_{r-k}(\mathbf{x})\to \cT_{r-k}(\mathbf{x})$ be the natural forgetful morphism. Then
\begin{align}
\label{eq:pomult}
\bbH_k(\mathbf{x}) & =    \sum_{\tilde\Gamma \in \cT^\fc_{r-k}(\mathbf{x})} \varphi(\tilde\Gamma)\prod_{v}(\val(v)-2)\ \Delta_{f(\tilde\Gamma)} \nonumber\\
 & =\sum_{\Gamma \in  \cT_{r-k}(\mathbf{x})}  m_\fc(\Gamma)\varphi(\Gamma) \prod_{v}(\val(v)-2)\Delta_{\Gamma}.
\end{align}
\end{lemma}
\begin{proof}
We observed that  for each  boundary stratum  $\Delta_{\Gamma}$ in $\mon$ there are precisely $ m_\fc(\Gamma)$ boundary strata in $\Htk$ pushing forward to it. It remains to show that each such stratum $[S]$ pushes forward with multiplicity $ \varphi(\Gamma)\prod_{v}(\val(v)-2)$.  We obtain this multiplicity by adapting formula \eqref{bound:multi} to the simple shape of the boundary strata we are observing. Over each $T_i$ in the chain there is only one connected component  of $X$ that maps non-trivially. Therefore we can replace the moduli space of possibly disconnected covers $\mathcal{M}^\bullet_i$ of \eqref{bound:multi} by the unique moduli space of non-trivial connected covers to $T_i$, call it $\mathcal{M}_i$, times an automorphism factor of $1/y_i^j$ for every connected component of degree $y_i^j$ mapping  trivially to $T_i$; we denote by $\triv_i$ the number of  trivial covers of the $i$-th component. Also the factors of $1/|\Aut(\mathbf{z_i})|$ need to be replaced by the automorphisms of each sub-partition of $\mathbf{z_i}$ corresponding to nodes on the same pair of curves on the two sides of the $i$-th node of $T^{r-k}$. But $X$ is a rational curve and hence all such sub-partitions have length one and trivially no automorphisms.  Hence \eqref{bound:multi} becomes:

\begin{equation}
\label{eq:popo}
[S]= \frac{ \prod_{i=1}^{r-k-1}\prod_{j=1}^{r_i} z_i^j
}{ \prod_{i=1}^{r-k}\prod_{j=1}^{{\triv}_i} y_i^j}
 \left[\gl_\ast\left( \prod_{i=1}^{r-k} \mathcal{M}_i \right)\right].
\end{equation}

Equation \eqref{eq:pomult} is deduced from \eqref{eq:popo} via the following observations:
\begin{enumerate}
\item There is a factor of  $y_i^j$ for each $2$-valent vertex  of the pre-stable tropical dual of $[S]$, and $y_i^j$ is equal to the weight of (either) edge on each side of the vertex. Therefore all $y_i^j$'s cancel with some of the $z$'s. The remaining multiplicity from the first part of \eqref{eq:popo} is therefore the product of weights of all internal edges of the tropical dual graph: by definition this is   $\varphi(\Gamma)$.
\item Each $\mathcal{M}_i$ is a moduli space of connected, rubber relative stable maps, with one simple transposition marked. By our discussion in Section \ref{sec:mscm}, we know  $\st$ maps  $\mathcal{M}_i$ onto $\mon$ with degree $r={{r}\choose{1}}$. If we call $v_i$ the vertex of the tropical dual graph corresponding to $\mathcal{M}_i$, then in this case $n= \val (v_i)$ and thus $r=\val (v_i)-2$.
\end{enumerate}
\end{proof}

To obtain Theorem \ref{thm:poly} it is now sufficient to remark that the weights of the edges of the tropical dual graph are linear homogeneous polynomials in the $x_i$'s and there are precisely $n-3-k$ internal edges in the tropical dual graph of any stratum  in $\Htk$ that pushes forward non-trivially to $\mon$.

We conclude this section by noting that our arguments do not apply exclusively to Hurwitz cycles, but to any cycle obtained by pull-pushing a boundary stratum from Losev-Manin to $\mon$. We thus obtain the following:

\begin{cor}
Let $\mathbf{x} \in \fc$  and $\overline\Delta$ be a union of boundary strata of dimension $k$ in $\LM$. Then $\st_\ast \br^\ast(\overline\Delta)$ is a homogeneous polynomial class of degree $n-3-k$ in $\fc$.
\end{cor}

\section{Tropical Hurwitz Loci}

\label{sec:thl}

\begin{defn}Let $\mathbf{x}\in \mathcal{H}$ and $r=n-2$. We define the $k$-dimensional tropical Hurwitz cycle as
$$\bbH^{\trop}_k(\mathbf{x}):= \ft_{\ast}\Big(\Psi_{n+1}\cdot \prod_{i=n+2}^{n+r-k} \big(\Psi_i \cdot \ev_i^{\ast}(p_i) \big) \cdot \calM_{0,r-k}(\TP,\mathbf{x})\Big) \subset \calM_{0,n},$$
where the $p_i$ are fixed arbitrary points in $\mathbb{R}\subset \TP$. (The notions of tropical moduli spaces, intersection theory and Psi-classes we use here are as recalled in Section \ref{sec:tg}.) We also define the Hurwitz locus:
$$\tilde{\bbH}^{\trop}_k(\mathbf{x}):= \Psi_{n+1}\cdot \prod_{i=n+2}^{n+r-k} \big(\Psi_i \cdot \ev_i^{\ast}(p_i) \big) \cdot \calM_{0,r-k}(\TP,\mathbf{x}).$$
\end{defn}

We describe the (general) tropical curves parameterized by $\bbH^{\trop}_k(\mathbf{x})\subset \calM_{0,n}$: if the $p_i$'s and  $0$ are distinct then the evaluation pullbacks force the contracted ends $n+1,\ldots,n+r-k$ to be adjacent to distinct vertices. Each Psi-class makes the corresponding contracted end adjacent to a $4$-valent vertex.  All contracted ends are forgotten in the push-forward. Thus the top-dimensional cells of  $\bbH^{\trop}_k(\mathbf{x})$ consist of  tropical curves that admit a cover of degree $\mathbf{x}$ to $\TP$ such that $r-k$ of its $r$ $3$-valent vertices are mapped to the points $p_i$ and $0$.

\begin{rem}
If $k=0$, all $r$ $3$-valent vertices are fixed. 
Since the evaluation maps at the contracted ends in the intersection of Psi-classes coincides with the tropical branch map (see \cite{CJM10}*{Definition 5.14}),
each point in the $0$-dimensional cycle $\bbH^{\trop}_0(\mathbf{x})$ comes with a multiplicity which equals the multiplicity of the tropical branch map. The degree of the $0$-dimensional cycle $\bbH^{\trop}_0(\mathbf{x})$ thus equals $|\Aut(\mathbf{x})|$ times the tropical Hurwitz number $H^0(\mathbf{x})$ as  defined in \cite{CJM10}. 
\end{rem}

\begin{rem}\label{rem:equivalent}
The cycle $\bbH^{\trop}_k(\mathbf{x})$ depends on the choice of the $p_i$.
 By \cite{AR07} however, different choices of $p_i$ yield rationally equivalent cycles. 
\end{rem}

\begin{lemma}\label{lem-evproduct}
 Let $\tilde{\Gamma}$ be a combinatorial type of a cover in  $\calM_{0,r-k}(\TP,\mathbf{x})$ with $r-k$ vertices of valence at least $4$ and with each contracted end $n+1,\ldots,n+r-k$ adjacent to one of the $r-k$ vertices. Then the evaluation map $\ev_{n+2}\times \ldots \times \ev_{n+r-k}$ in local coordinates of the cone of $\tilde{\Gamma}$ is a square matrix such that the absolute value of its determinant equals the product of the weights of the bounded edges of $\tilde{\Gamma}$.
\end{lemma}
\begin{proof}
 Inductively we can order the vertices which are not adjacent to $n+1$ and the bounded edges of $\tilde{\Gamma}$ so that going from end $n+1$ to another contracted end
 only bounded edges with a smaller index are traversed
 (see also \cite{CJM10}*{Lemma 5.26}). Then the evaluation map in local coordinates is a square triangular matrix such that the diagonal entries are the weights of the bounded edges of $\tilde{\Gamma}$.
\end{proof}

\begin{defn}\label{def:weights} Let $\mathbf{x}\in \mathcal{H}$, and
 $\Gamma$  a combinatorial type of a marked tropical curve in $\calM_{0,n}$. Associate to each end $i$ of $\Gamma$ the weight $x_i$. The weights of the other edges are determined by the balancing condition (see Lemma \ref{lem:flat}). Orient the edges of $\Gamma$ so that all weights are positive. The orientation of the edges induces a partial ordering of the vertices of $\Gamma$.
For any choice of $k$ vertices $V_1,\ldots,V_k$ of $\Gamma$ we denote by $m_{V_1,\ldots,V_k}(\Gamma)$ the number of ways to totally order the vertices different from $V_1,\ldots,V_k$ of $\Gamma$ respecting the partial order induced by the orientation (see also \cite{CJMwc}*{Definition 2.1 and 2.15}). (If $k=0$, this agrees with $m(\Gamma)$ from Definition \ref{mult}.) We call the vertices $V_1,\ldots,V_k$ which are not ordered the \emph{moving vertices}.

For a fixed choice of $k$ moving vertices, let $\Gamma'$ be a graph obtained from $\Gamma$ by shrinking edges in such a way that all moving vertices are identified with a non-moving vertex, and no pair of non-moving vertices is identified. We denote by $w_{V_1,\ldots,V_k}(\Gamma)$ the greatest common divisor, over all possibilities to choose $\Gamma'$, of the products of the weights of the bounded edges of $\Gamma'$.
\end{defn}

\begin{defn}
 For a combinatorial type $\tilde{\Gamma}$ of tropical covers and a vertex $V$ of $\Gamma$, we define the \emph{moving vector} of $V$ to be the vector in $\calM_{0,N}$ that we have to add to a tropical cover of type $\Gamma$ in order to move $h(V)$ to the right by an interval of legths $\prod w(e)$ where the product goes over all bounded edges $e$ adjacent to $V$ and $w(e)$ denotes their weight.
The moving vector  corresponds to a tree with nonzero lengths for the bounded edges of $\Gamma$ adjacent to $V$: each edge $e$ pointing out of $V$ has  length $-\prod_{e'\neq e}w(e')$, each edge pointing into $V$ has length $\prod_{e'\neq e}w(e')$; in both cases the product is over bounded edges $e'$ adjacent to $V$ and not equal to $e$ (see \cite{GMO}*{Definition 5.5}).
We also call the image of a moving vector in $\calM_{0,n}$ under the forgetful map $\ft$ a moving vector.
\end{defn}

\begin{lemma}\label{lem:cellsweights} Let $\Gamma$ be a combinatorial type of a top-dimensional cone of $\calM_{0,n}$.
If the points $p_i$ are chosen generically (i.e.\ pairwise distinct and different from $0$), there are $\sum m_{V_1,\ldots,V_k}(\Gamma)$ top-dimensional cells of  $\bbH^{\trop}_k(\mathbf{x})$ contained in the cone of $\Gamma$, where the sum goes over all $\binom{r}{k}$ choices of $k$ moving vertices $V_1,\ldots,V_k$ of $\Gamma$. For a fixed choice of $k$ moving vertices, each of the $m_{V_1,\ldots,V_k}(\Gamma)$ cells corresponding to this choice has weight $w_{V_1,\ldots,V_k}(\Gamma)$. Moreover, all cells corresponding to a fixed choice of $k$ moving vertices are parallel, i.e.\ their linear part is the same.
\end{lemma}
\begin{proof}
 A cell of $\tilde{\bbH}^{\trop}_k(\mathbf{x})$ maps to the cone of $\Gamma$ in $\calM_{0,n}$ if we obtain the combinatorial type $\Gamma$ after forgetting the ends $n+1,\ldots,n+r-k$. 
Thus to recover such a cell from the graph $\Gamma$ one must to pick $r-k$ vertices to which we attach the marked ends $n+1,\ldots,n+r-k$ in such a way that we get a cover mapping the end $i$ to the point $p_i$ for $i=n+2,\ldots,n+r-k$ and end $n+1$ to $0$. One can do this if and only if the orientation of the edges 
respects the order of the image points $p_i$ of the $r-k$ vertices to which the mark ends are attached. 

The graphs $\Gamma'$ in Definition \ref{def:weights}
 correspond to faces of the boundary of the cone of $\Gamma$ that each cell corresponding to $V_1,\ldots,V_k$ intersects in a point. 
Thus two cells of $\tilde{\bbH}^{\trop}_k(\mathbf{x})$ corresponding to different choices of moving vertices are pushed forward to different cells of $\bbH^{\trop}_k(\mathbf{x})$ by $\ft$ --- their images under $\ft_\ast$ intersect different boundary cones of $\Gamma$ in points. But also two cells corresponding to the same choice of moving vertices are pushed forward to different cells: they intersect the same boundary cones of $\Gamma$, but they do so in different points as the coordinates of the intersection points are given by the lengths of the bounded edges of each $\Gamma'$, and these lengths depend on the choice of which of the $r-k$ non-moving vertices is mapped to which of the $p_i$.

It follows that there are $\sum m_{V_1,\ldots,V_k}(\Gamma)$ top-dimensional cells of  $\bbH^{\trop}_k(\mathbf{x})$ contained in the cone of $\Gamma$.

We now compute the weight of one such cell. We start by computing the weight of a cell in $\tilde{\bbH}^{\trop}_k(\mathbf{x})$. Each cell of $\prod_{i=n+1}^{n+r-k}\Psi_i$ corresponds to a combinatorial type where the ends $n+1,\ldots,n+r-k$ are adjacent to different $4$-valent vertices. The weight of such a cone in the intersection $\prod_{i=n+1}^{n+r-k}\Psi_i$ equals one by \cite{KM07}. 
By Remark \ref{rem-intersect} the intersection multiplicity from the evaluation pullbacks is the gcd of the absolute values of the maximal minors of the matrix of the evaluation map in local coordinates. 
We evaluate each of the $r-k-1$ non-moving vertices different from $n+1$ (which is sent to $0$ by convention). A maximal minor corresponds to a choice of $r-k-1$ bounded edges: we contract the remaining bounded edges and get a graph $\Gamma'$ with $r-k$ vertices.
 If some of the non-moving vertices  in $\Gamma'$  are identified, then the minor has some identical rows (resp.\ a zero row) and is thus zero. 
 All the $r-k$ non-moving vertices must thus remain separate; each of the moving vertices is identified with some non-moving vertex, as in the definition of $w_{V_1,\ldots,V_k}(\Gamma)$. 
Finally Lemma \ref{lem-evproduct} implies that the absolute value of the maximal minor corresponding to a possible choice of $\Gamma'$ equals the product of the weights of the bounded edges of $\Gamma'$. 
 The weight of a cell of $\tilde{\bbH}^{\trop}_k(\mathbf{x})$ equals $w_{V_1,\ldots,V_k}(\Gamma)$. The push forward maps cells of this intersection to cells of $\bbH^{\trop}_k(\mathbf{x})$ one by one, and the forgetful map has index one since it is just a projection. Thus the claim about the weight of cells of $\bbH^{\trop}_k(\mathbf{x})$ follows.

For any cell  $\tilde{\Gamma}\in \tilde{\bbH}^{\trop}_k(\mathbf{x})$ where $V_1,\ldots,V_k$ are the moving vertices, the linear part of the affine hull of the cell $\alpha$ is  spanned by the moving vectors of the vertices $V_1,\ldots,V_k$ in $\calM_{0,N}$. Thus all cells of $\bbH^{\trop}_k(\mathbf{x})$ in the cone of $\Gamma$ corresponding to a fixed choice of $k$ moving vertices are parallel.
\end{proof}

Lemma \ref{lem:cellsweights} implies that a (top-dimensional) cell of $\bbH^{\trop}_k(\mathbf{x})$ is specified by a (3-valent) graph $\Gamma$, a choice of $k$ moving vertices and an ordering of the remaining vertices that respects the edge orientations given by $\mathbf{x}$.
We call this data  the \emph{combinatorial type} of a cell of $\bbH^{\trop}_k(\mathbf{x})$.

\begin{rem}
 $\bbH^{\trop}_k(\mathbf{x})$ is not necessarily a simplicial polyhedral complex. Figure \ref{fig:notsimplic} shows a combinatorial type corresponding to a two-dimensional cell. The moving vertices are labelled with an $M$. This cell has four vertices, since the moving vertices can merge with 1 and 3, or 1 and 2, or 2 and 3, or both with 2.

\begin{figure}[tb]
\input{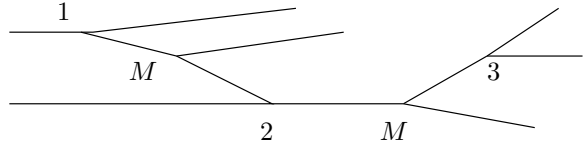}
\caption{The combinatorial type of a twodimensional cell with four vertices.}
\label{fig:notsimplic}
\end{figure}

\end{rem}

Since we define $\bbH^{\trop}_k(\mathbf{x})$ as a tropical intersection product, it is of course a tropical cycle itself and thus in particular balanced around each cell of codimension one. We find it useful to have a concrete local description of the balancing of tropical Hurwitz cycles. We conclude this section by discussing this.


The Hurwitz cycle $\bbH^{\trop}_k(\mathbf{x})$ has the structure of a weighted marked polyhedral complex (see \cite{GMO}*{Definition 5.1}): let $\sigma$ denote a top-dimensional cell of $\bbH^{\trop}_k(\mathbf{x})$ and $\tilde{\sigma}$ its unique preimage in $\tilde{\bbH}^{\trop}_k(\mathbf{x})$. Let $\tilde{\Gamma}$ be the combinatorial type of the cell $\tilde{\sigma}$. 
We denote by $\Gamma$ the image of $\tilde{\Gamma}$ under the forgetful map. A bounded edge of $\Gamma$ is called \emph{a fixed edge}, if the two end vertices of its preimage in $\tilde{\Gamma}$ are adjacent to two contracted ends (i.e.\ their image is fixed by the evaluation pullback). We call a bounded edge \emph{moveable}, if neither of the two end vertices of its preimage is adjacent to a contracted end.
We define the weight factor $\omega'(\sigma)$ to be the product of the weights of all fixed edges divided by the product of the weights of moveable edges.
The $k$ moving vectors of the moving vertices are the marked vectors of $\sigma$. 

We equipped the polyhedral complex underlying $\bbH^{\trop}_k(\mathbf{x})$ with two a priori different weight functions for its top-dimensional cells: one weight function is defined by the tropical intersection product, the other by the structure of a marked polyhedral complex. It can be shown that these two structures agree. Assuming this, we now check the balancing condition (as in \cite{GMO}*{Lemma 5.2}) locally around a codimension one face.

A cell $\tau$ of codimension one of $\bbH^{\trop}_k(\mathbf{x})$ 
parameterizes trees with exactly one $4$-valent vertex. There are two types of cells of codimension one: those where two moving vertices have merged into a $4$-valent vertex, and those where a moving vertex has merged with a non-moving vertex.

When two moving vertices have merged, there are three top-dimensional neighbors of $\tau$, corresponding to the three resolutions of the $4$-valent vertex.

The three neighboring cones  share $k-2$ marked vectors: the remaining two marked vectors  (for each cone) correspond to the moving vectors of the  new $3$-valent vertices obtained resolving the $4$-valent vertex. 
In all three cases we can replace one of the two new moving vectors by the moving vector of the $4$-valent vertex of $\tau$. Now the hypothesis of Proposition 5.2 in \cite{GMO} apply, and we can show just as in Proposition 5.8 of \cite{GMO} that the balancing condition at $\tau$ is equivalent to  the appropriate weighted sum of the three other new moving vectors equal to the moving vector of the $4$-valent vertex. 

Assume now that a moving vertex is merged into a non-moving vertex to form the $4$-valent vertex $V$ of $\tau$.
Without restriction, assume that one bounded edge of weight $a$ points into $V$ and three bounded edges of weights $b$, $c$ and $d$ point out of $V$.
 Then there are six top-dimensional neighbors of $\tau$, illustrated in Figure \ref{fig:neighbours}.
All but one of the marked vectors coincide for the six top-dimensional neighbors, they only differ by the moving vector of the new vertex. The moving vertex is indicated with an $M$.
The weight factors of the six neighboring cells differ only by two of the four bounded edges adjacent to $V$ which are fixed edges according to where the new moving vertex is.

\begin{figure}[tb]
\input{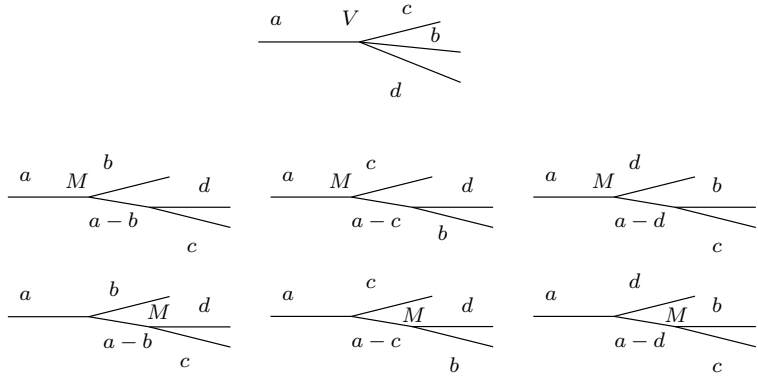}
\caption{The six neighbors of a cell of codimension one where a moving vertex is merged with a non-moving vertex.}
\label{fig:neighbours}
\end{figure}

We apply Proposition 5.2 of \cite{GMO} and deduce that the balancing condition around $\tau$ is equivalent to the equation: 
\begin{align*}
&ab\cdot (cd\cdot v_{ab}- (a-b)d \cdot v_c - (a-b) c \cdot v_d)\\
&-  cd\cdot (-ab\cdot v_{ab}+ (a-b)b \cdot v_a - (a-b) a \cdot v_b)\\
&+ac\cdot (bd\cdot v_{ac}- (a-c)d \cdot v_b - (a-c) b \cdot v_d)\\
&-bd\cdot (-ac\cdot v_{ac}+ (a-c)c \cdot v_a - (a-c) a \cdot v_c)\\
&+ad\cdot (bc\cdot v_{ad}- (a-d)c \cdot v_b - (a-d) b \cdot v_c)\\
&-bc\cdot (-ad\cdot v_{ad}+ (a-d)d \cdot v_a - (a-d) a \cdot v_d)
\\&= 2abcd\cdot (v_{ab}+v_{ac}+v_{bc}-(v_a+v_b+v_c+v_{abc}))=0,
\end{align*}
where we use the equation $a=b+c+d$ and relations among the vectors $v_I$ as in \cite{KM07}*{Lemma 2.6} (the vector $v_a$ here stands for the tree with all ends that we can reach from $V$ via $a$ on one side).

\subsection{Extended Example: Hurwitz Curves}
We study the tropical structure of the Hurwitz curve $\bbH^{\trop}_1(\mathbf{x})$. Top dimensional cells (segments) parameterize graphs with one moving vertex.
Each codimension one cell --- i.e.\ vertex --- of the Hurwitz curve
corresponds to curves where the moving vertex is merged with another (fixed) vertex, and has six top-dimensional cells incident to it. For any $\mathbf{x}$, the Hurwitz curve is a six-valent tropical curve in  $\calM_{0,n}$. We  distinguish several types of edges and compute their weights $w_{V_1}(\Gamma)$.

\begin{figure}[tb]
\input{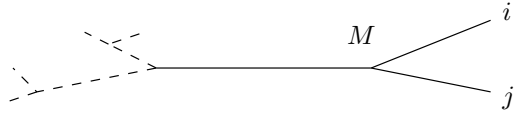}
\caption{The combinatorial type of a constant end of the Hurwitz curve.}
\label{fig:easyend}
\end{figure}

\emph{Constant Ends} (Figure \ref{fig:easyend}) of the Hurwitz curve are edges corresponding to a combinatorial type where the moving vertex is adjacent to two ends $i$ and $j$. The moving vector equals $\pm v_{ij}$. 
A constant end is incident to a unique vertex, corresponding to shrinking the only bounded edge adjacent to the moving vertex. 
The weight of a constant end equals the product of all bounded edge weights except for the bounded edge adjacent to the moving vertex. 


\begin{figure}[tb]
\input{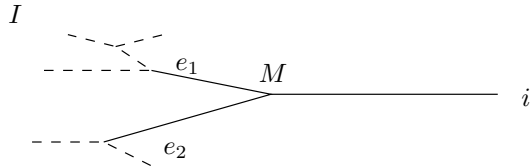}
\caption{The combinatorial type of a linear end of the Hurwitz curve.}
\label{fig:complicatedend}
\end{figure}

\emph{Linear Ends} (Figure \ref{fig:complicatedend}):
 the moving vertex is adjacent to two bounded edges $e_1$ and $e_2$ of the same direction and one end $i$, necessarily of opposite direction. The moving vector of a linear end equals $\pm( w(e_2) v_{I}+w(e_1) v_{I\cup \{i\}})$, where $I$ denotes the set of ends separated from $i$ via $e_1$.
 A linear end is incident to only one $0$-dimensional cell, corresponding to the graph where the moving vertex coincides with the vertex of the shorter of the two edges (the end vertex of $e_1$ in the picture). To obtain possible graphs $\Gamma'$, we can shrink either $e_1$ or $e_2$. Thus the weight of a linear end equals $\prod_{e\neq e_1,e_2} w(e)\cdot \mbox{gcd}(w(e_1),w(e_2))$ where the product goes over all bounded edges $e$ except $e_1$ and $e_2$.

\begin{figure}[tb]
\input{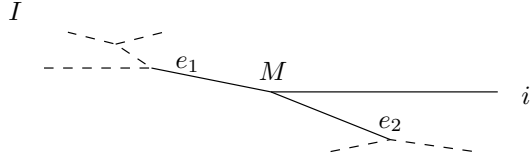}
\caption{The combinatorial type of a linear edge of the Hurwitz curve.}
\label{fig:easybounded}
\end{figure}

\emph{Linear Edges} (Figure \ref{fig:easybounded}) are edges corresponding to a combinatorial type where the moving vertex is adjacent to two bounded edges $e_1$ and $e_2$ of opposite direction, and to an end $i$. The moving vector of an linear edge equals $\pm( w(e_2) v_{I}- w(e_1) v_{I\cup \{i\}})$ where $I$ denotes the set of ends separated from $i$ via $e_1$. We reach the two vertices of a linear edge by shrinking either $e_1$ or $e_2$. The weight of a linear edge is exactly as for a linear end: $\prod_{e\neq e_1,e_2} w(e)\cdot \mbox{gcd}(w(e_1),w(e_2))$.

\begin{figure}[tb]
\input{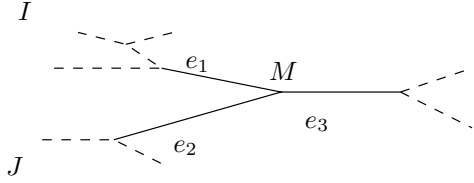}
\caption{The combinatorial type of a quadratic edge of the Hurwitz curve.}
\label{fig:complicatedbounded}
\end{figure}

\emph{Quadratic Edges} (Figure \ref{fig:complicatedbounded}) are edges corresponding to a combinatorial type where the moving vertex is adjacent to three bounded edges $e_1$, $e_2$ and $e_3$. Without restriction we can assume that $e_1$ and $e_2$ are of the same direction, and $e_3$ is of opposite direction. We can also assume that the non-moving vertices are ordered such that the other end vertex of $e_1$ is between the end vertex of $e_2$ and the moving vertex. A quadratic edge connects the two vertices where we shrink $e_3$ resp.\ $e_1$.
 The moving vector of a quadratic edge equals $\pm( w(e_2)\cdot w(e_3) v_{I}+ w(e_1)w(e_3) v_{J}- w(e_1)w(e_2) v_{(I\cup J)^c})$ where $I$ denotes the set of ends that we can reach from the moving vertex via $e_1$ and $J$ denotes the set of ends that we can reach via $e_2$. 
We can shrink each of the three adjacent edges to obtain a possible graph $\Gamma'$ to compute the weight. Thus the weight of a complicated bounded edge equals $\prod_{e\neq e_1,e_2,e_3} w(e)\cdot \mbox{gcd}(w(e_1)w(e_2),w(e_1)w(e_3),w(e_2)w(e_3))$.

\begin{rem}
The weight of a one-dimensional cell $H$ in the Hurwitz curve is a function of the weights of the bounded edges of a curve $\tilde{\Gamma}$ parameterized by the cell. If we artificially define the weight of ends of  $\tilde{\Gamma}$ to be $1$ (these are not the right tropical weights!) then we can describe the weight of $H$ in a uniform way as the product of weights of all edges of  $\tilde{\Gamma}$ that are not incident to the moving vertex times the $gcd$ of the three products of  weights of pairs of edges incident to the moving vertex. Given that  weights of edges  of  $\tilde{\Gamma}$ are either $1$ or linear polynomials, the degree of the above $gcd$ can range from $0$ to $2$, thus explaining the names we chose above.\end{rem}

\begin{example}\label{ex:m05hurwitzcurve}
 We compute the Hurwitz curve $\bbH^{\trop}_1(\mathbf{x})$ for $\mathbf{x}=(x_1,
\ldots,x_5)$ with $x_1, x_2>0$, $x_3,x_4,x_5<0$, $x_1>|x_i+x_j|$ for all $i\not=j\in\{3,4,5\}$ and $x_2<-x_i$ for all $i=3,\ldots,5$ in $\calM_{0,5}$.
We start with the cells in the cone spanned by $v_{12}$ and $v_{34}$ in $\calM_{0,5}$. To fix a combinatorial type of covers such that the corresponding cell of $\bbH^{\trop}_1(\mathbf{x})$ lives in this cone, we first have to choose a moving vertex among the three vertices of the tree with ends $1$ and $2$ coming together and ends $3$ and $4$ coming together. For each of these three choices, there is one ordering of the remaining vertices which is compatible with the orientation of the edges (see Figure \ref{fig:m05}). We thus have three cells of $\bbH^{\trop}_1(\mathbf{x})$ in this cone, two of these cells are constant ends, the other one is an linear edge. Figure \ref{fig:m052} shows the cone and the three cells.
By symmetry, the cones spanned by $v_{12}$ and $v_{35}$ resp.\ by $v_{12}$ and $v_{45}$ look analogous. Similar arguments also show that all remaining cones except the cones spanned by $v_{23}$ and $v_{45}$, resp.\  $v_{24}$ and $v_{35}$, resp.\  $v_{25}$ and $v_{35}$, look alike.

\end{example}
\begin{figure}[tb]
\input{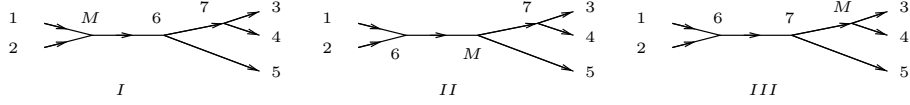}
\caption{The combinatorial types of the Hurwitz curve in the cone of $\calM_{0,5}$ spanned by $v_{12}$ and $v_{34}$.}
\label{fig:m05}
\end{figure}

\begin{figure}[tb]
\input{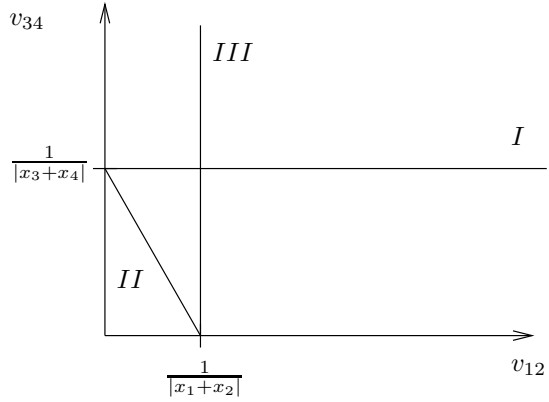}
\caption{The Hurwitz curve in the cone of $\calM_{0,5}$ spanned by $v_{12}$ and $v_{34}$.  We require the vertex $6$ to be mapped to $0$ as usual and $7$ to $1$.}
\label{fig:m052}
\end{figure}

In the cone spanned by $v_{23}$ and $v_{45}$, we have two constant ends when the moving vertex is adjacent to two ends. If the third vertex is moving, there are two orderings of the remaining vertices compatible with the orientation of the edges (see Figure \ref{fig:m053}). Altogether, we get two constant and two linear ends as depicted in Figure \ref{fig:m054}.

\begin{figure}[tb]
\input{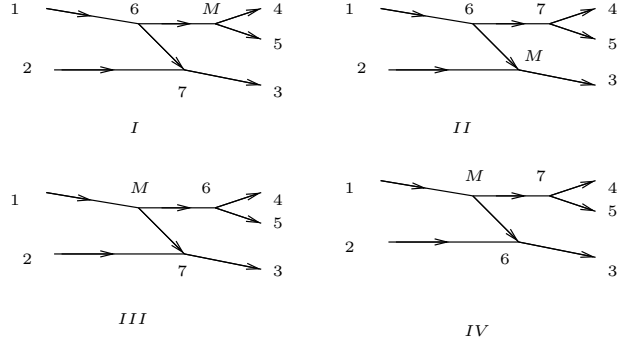}
\caption{The combinatorial types of the Hurwitz curve in the cone of $\calM_{0,5}$ spanned by $v_{23}$ and $v_{45}$.}
\label{fig:m053}
\end{figure}

\begin{figure}[tb]
\input{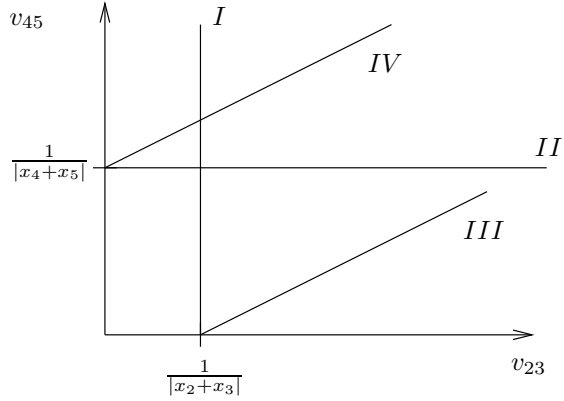}
\caption{The Hurwitz curve in the cone of $\calM_{0,5}$ spanned by $v_{23}$ and $v_{45}$.}
\label{fig:m054}
\end{figure}

\section{Tropical-Classical Correspondence}
\label{sec:tcc}

Our work in Sections \ref{sec:hl} and \ref{sec:thl} has highlighted a combinatorial correspondence between  classical and tropical Hurwitz cycles. In this section we make a precise statement of such a correspondence, and illustrate it in the example of one dimensional cycles. A satisfactory correspondence should also encode the polynomial multiplicities of strata. In Section \ref{sec:int} we interpret the (classical) multiplicity of a stratum in the Hurwitz cycle as an intersection multiplicity of the corresponding tropical face with the  $k$-dimensional skeleton of $\calM_{0,n}$.
\begin{co}
\label{co:tropclas}
There is a natural bijection between $i$-dimensional faces of $\bbH^{\trop}_k(\mathbf{x})$ and connected components in  $\tilde\bbH_k(\mathbf{x})$ of the inverse image via $\st$ of irreducible  strata in $\bbH_k(\mathbf{x})$ of dimension $k-i$. Further, incidence of faces on the tropical side corresponds to intersection of strata on the classical side.  
\end{co}

The key ingredient here is the correspondence between tropical graphs and boundary strata of moduli spaces of relative stable maps outlined in Lemma \ref{flatten}. The subtlety to observe is that a cell of the tropical Hurwitz cycle is sensitive to the type of the directed graph parameterized and to the ordering of the fixed vertices, but not to the ordering of moving vertices amongst themselves or with respect to the fixed ones. So two general points  $P_1,P_2$ in the same $i$-dimensional tropical cell may correspond to graphs where two adjacent vertices (at least one of which is moving) have switched order, and hence to different $(k-i)$-dimensional boundary strata $\tilde{\Delta}_1, \tilde{\Delta}_2$ of relative stable maps. The segment joining $P_1$ and $P_2$ contains a point $P$ where the two incriminated vertices map to the same image point: this corresponds to a $(k-i+1)$-dimensional boundary stratum $\Delta$ that contains $\tilde{\Delta}_1$ and $\tilde{\Delta}_2$ as specializations. The stratum $\Delta$ does not belong to $\tilde{\mathbb{H}}_k(\mathbf{x})$, but it does belong to the inverse image via $\st$ of the Hurwitz cycle.  We feel that describing this phenomenon in full generality would only mire us in notational confusion, so we choose  to illustrate it in one very specific example.

\subsection{Hurwitz Curves Continued}

Consider the one dimensional cell labelled $I$ in Figure \ref{fig:m054}, and the corresponding graphs parameterized by points in such cell (these are illustrated in Figure \ref{fig:m053}). Let $\ell$ be a coordinate for this cell, corresponding to the length of the segment joining vertices $6$ and $M$. For $\ell < -\frac{1}{x_4+x_5}$ (resp. $\ell > -\frac{1}{x_4+x_5}$) any point of $I$ is the tropical dual graph to the stratum $\tilde{\Delta}_1$ (resp. $\tilde{\Delta}_2$) as depicted in Figure \ref{fig:contract}. The point $\ell = -\frac{1}{x_4+x_5}$ correspond to the stratum $\tilde{\Delta}$, which is a $\proj$ connecting $\tilde{\Delta}_1$ and  $\tilde{\Delta}_2$ in $\st^{-1}({\mathbb{H}_1}(\mathbf{x}))$.

\begin{figure}[tb]
\input{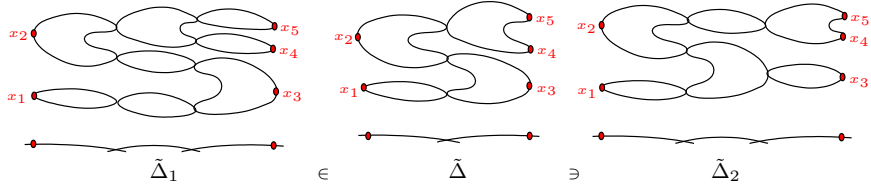}
\caption{Strata of relative stable maps in $\st^{-1}({\mathbb{H}_1}(\mathbf{x}))$.}
\label{fig:contract}
\end{figure}

\subsection{The Hurwitz cycle 
intersecting the codimension $k$-skeleton of $\calM_{0,n}$}

\label{sec:int}



In this section we realise the (classical) multiplicities of the strata in the Hurwitz cycle as intersection numbers of the tropical Hurwitz cycle with the codimension  $k$-skeleton of $\calM_{0,n}$. To do so we must  view each codimension $k$ cell as a part of an intersection product of divisors. 
We recall the boundary divisors $D_I$  that ``play well'' in the intersection theory of  $ \calM_{0,n}$ are defined as divisors of appropriate rational functions \cite{Rau08}*{Definition 2.4}. Each $D_I$ is a linear combination of codimension one cells of $ \calM_{0,n}$ with appropriate weights. In the following lemma we describe some intersection products of these boundary divisors in term of the tropical curves they parameterize.
\begin{lemma}\label{lem:divintersect}
The intersection of the tropical boundary divisors $ D_{12}\cdot D_{123}\cdot  \ldots \cdot D_{1...j}$ for some $j\geq 2$ in $\calM_{0,m}$ consists of all cones corresponding to a type with the ends $1,\ldots,j$ at
\begin{itemize}
 \item  a $j+2$-valent vertex and only $3$-valent vertices otherwise with weight one,
\item  a $j+1$-valent vertex adjacent to a $4$-valent vertex and only $3$-valent vertices otherwise with weight $-1$.
\end{itemize}
\end{lemma}
\begin{proof}
We show this by induction. The induction beginning is \cite{Rau08}*{Lemma 2.5}. For the induction step, assume that the statement is true for $j-1$. 
The codimension $j$ cones in  $ D_{12}\cdot D_{123}\cdot  \ldots \cdot D_{1...j-1}$ then consist of all cones corresponding to a type with the ends $1,\ldots,j-1$ at
\begin{itemize}
 \item a $j+2$-valent vertex and only $3$-valent vertices otherwise,
\item  a $j+1$-valent vertex and one $4$-valent vertex,
\item  a $j$-valent vertex adjacent to a $5$-valent vertex,
\item  a $j$-valent vertex adjacent to a $4$-valent vertex and one other $4$-valent vertex.
\end{itemize}
To obtain the coefficients of such cones in the product   $ D_{12}\cdot D_{123}\cdot  \ldots \cdot D_{1...j}$ we must compute the intersection with $\varphi_{1...j}$ around each of these cones.

In the first case, the neighbors of  $ D_{12}\cdot D_{123}\cdot  \ldots \cdot D_{1...j-1}$ are the resolutions of the $j+2$-valent vertex in a $j+1$-valent vertex with $1,\ldots,j-1$ adjacent. Each such neighbor is spanned by a vector corresponding to this resolution. Since only the vector $v_{1...j}$ is mapped to one by $\varphi_{1...j}$, we get a contribution of one for the weight of the codimension one cone only if also $j$ is adjacent to the $j+2$-valent vertex. The weight then equals one.

In the second case, we can resolve the $j+1$-valent vertex, but none of the resolutions is spanned by the vector $v_{1..j}$. We can also resolve the $4$-valent vertex. Again, none of the resolutions is spanned by the vector $v_{1...j}$, however, this vector is contained in the codimension one cone if also $j$ is adjacent to the $j+1$-valent vertex. If in addition the $4$-valent vertex is adjacent to the $j+1$-valent vertex, the sum of the vectors spanning its three resolutions contains $v_{1...j}$ as a summand: if we denote the four edges adjacent to the $4$-valent vertex by $e_1,\ldots,e_4$ and assume that the subset of ends that can be reached via $e_i$ from the $4$-valent vertex is $A_i$, then the three resolutions are spanned by $v_{A_1\cup A_2}$, $v_{A_1\cup A_3}$ and $v_{A_2\cup A_3}$ and their sum satisfies $v_{A_1\cup A_2}+v_{A_1\cup A_3}+v_{A_2\cup A_3}=v_{A_1}+v_{A_2}+v_{A_3}+v_{A_1\cup A_2 \cup A_3}= v_{A_1}+v_{A_2}+v_{A_3}+v_{A_4}$ by \cite{KM07}*{Lemma 2.6}. This yields a contribution of minus one for the weight of this cone.

In the third and fourth case, neither the codimension one cone itself nor any of its neighbors contains the vector $v_{1...j}$. Therefore it gets weight zero. 
The claim follows.
\end{proof}

\begin{rem}\label{rem:coneintersect}
The important consequence of this lemma is the statement that a cone with a $j+2$-valent vertex adjacent to the ends $1,\ldots,j$ appears with weight one in the intersection $ D_{12}\cdot D_{123}\cdot  \ldots \cdot D_{1...j}$. It is straight-forward to generalize this statement to a cone $C$ with an arbitrary $j+2$-valent vertex $V$, adjacent to the edges $e_1,\ldots,e_{j+2}$.
We denote by $A_i$ the subset of ends that can be reached from $V$ via $e_i$. Then $C$ appears with weight one in the intersection $ D_{A_1\cup A_2}\cdot D_{A_1\cup A_2 \cup A_3}\cdot  \ldots \cdot D_{A_1\cup \ldots\cup A_{j-1}}$.
Also, to cut out a cone  with several higher-valent vertices, we can combine several such intersection products.
\end{rem}

\begin{lemma}\label{lem:pullbackintersect}
 The intersection $\Psi_\alpha\cdot \ft_\alpha^\ast( D_{12})\cdot\ft_\alpha^\ast( D_{123})\cdot   \ldots \cdot \ft_\alpha^\ast(D_{1...j})$ for some $j\geq 2$ in $\calM_{0,m+1}$ consists of all cones corresponding to a type with the ends $1,\ldots,j$ 
\begin{itemize}
 \item and $\alpha$ at a $j+3$-valent vertex with weight $j$, 
\item at a $j+2$-valent vertex, and $\alpha$ adjacent to some other vertex with weight $1$,
\item and $\alpha$ at a $j+2$-valent vertex adjacent to a $4$-valent vertex with weight $-(j-1)$,
\item at a $j+1$-valent vertex adjacent to a $5$-valent vertex with $\alpha$ with weight $-2$,
\item at a $j+1$-valent vertex adjacent to a $4$-valent vertex and $\alpha$ adjacent to some other vertex with weight $-1$.
\end{itemize}
\end{lemma}
\begin{proof}
 The proof is again by induction. For $j=2$, we intersect $\Psi_\alpha$ with $\ft_\alpha^\ast(\varphi_{12})=\varphi_{12}\circ \ft_\alpha$. This map sends the vector $v_{12}$ and the vector $v_{12\alpha}$ to one and all other $v_i$ to zero. Codimension one cones of $\Psi_\alpha$ either have a $5$-valent vertex with $\alpha$ or a $4$-valent vertex with $\alpha$ and another $4$-valent vertex. In the first case, if $1$ and $2$ are also adjacent to the $5$-valent vertex, two of the six neighbors are spanned by vectors mapping to one, so we get weight $2$. In the second case, the vector $v_{12\alpha}$ is contained in the codimension one cone itself and appears in the sum of the vectors spanning the neighbors if $1$ and $2$ are adjacent to the vertex with $\alpha$ and the other $4$-valent vertex is adjacent. Such a cone then comes with weight minus one.

For the induction step, assume the statement is true for $j-1$. The codimension one cones of $\Psi_\alpha\cdot \ft_\alpha^\ast( D_{12})\cdot\ft_\alpha^\ast( D_{123})\cdot   \ldots \cdot \ft_\alpha^\ast(D_{1...j})$ then consist of all cones corresponding to a type with the ends $1,\ldots,j-1$ 
\begin{enumerate}
 \item and $\alpha$ at a $j+3$-valent vertex, 
\item and $\alpha$ at a $j+2$-valent vertex and one $4$-valent vertex,
\item at a $j+2$-valent vertex, $\alpha$ somewhere else ,
\item at a $j+1$-valent vertex and a $5$-valent vertex with $\alpha$,
\item at a $j+1$-valent vertex and a $4$-valent vertex, $\alpha$ somewhere else,
\item and $\alpha$ at a $j+1$-valent vertex, next to a $5$-valent vertex,
\item and $\alpha$ at a $j+1$-valent vertex next to a $4$-valent vertex and another $4$-valent vertex,
\item at a $j$-valent vertex next to a $6$-valent vertex with $\alpha$,
\item at a $j$-valent vertex next to a $5$-valent vertex with $\alpha$ and another $4$-valent vertex,
\item at a $j$-valent vertex next to a $5$-valent vertex and $\alpha$ somewhere else,
\item at a $j$-valent vertex next to a $4$-valent vertex, and another $4$-valent vertex and $\alpha$ somewhere else. 
\end{enumerate}
For the cases (6)-(11) it is easy to see that neither the codimension one cone itself nor any of its neighbors contains the vectors $v_{1..j}$ or $v_{1..j\alpha}$ which are mapped to one by $\ft_\alpha^\ast(\varphi_{1..j})=\varphi_{1..j}\circ \ft_\alpha$. Thus any of these cones is taken with weight zero.

In the first case, if $j$ is also adjacent to the $j+3$-valent vertex, we have a neighbor spanned by $v_{1...j\alpha}$ with weight $j-1$, and a neighbor spanned by $v_{1...j}$  with weight one. Altogether the weight  is $j$.

In the second case, if $j$ is also adjacent to the $j+2$-valent vertex the vector $v_{1...j\alpha}$ is contained in the codimension one cone itself. If in addition the $4$-valent vertex is adjacent to the $j+2$-valent vertex, the vector $v_{1...j\alpha}$ appears in the sum of the three neighbors corresponding to the resolutions of the $4$-valent vertex. Since any such resolution comes with weight $j-1$, this codimension one cone has weight $-(j-1)$.

In the third case, if $j$ is also adjacent to the $j+2$-valent vertex, we have one neighbor spanned by $v_{1...j}$  with weight one, so we also get weight one for this cone.

In the fourth case, none of the resolutions of the $j+1$-valent vertex contains the vectors $v_{1...j}$ or $v_{1...j\alpha}$.
 The vector $v_{1...j}$ is contained in the cone itself however if $j$ is also adjacent to the $j+1$-valent vertex. We can also resolve the $5$-valent with $\alpha$ in such a way that $\alpha$ is still at a $4$-valent vertex. All these six neighbors have weight one. Denote the four edges not equal to the end $\alpha$ but adjacent to the $5$-valent vertex by $e_1,\ldots,e_4$ and denote by $A_i$ the subset of ends that can be reached from the $5$-valent vertex via $e_i$. Then the six neighbors are spanned by the vectors $v_{A_1\cup A_2}$, $v_{A_1\cup A_3}$, $v_{A_2\cup A_3}$, $v_{A_1\cup A_2\cup\{\alpha\}}$, $v_{A_1\cup A_3\cup\{\alpha\}}$ and $v_{A_2\cup A_3\cup\{\alpha\}}$ whose sum equals $2v_{A_1}+2v_{A_2}+2v_{A_3}+2\cdot v_{A_1\cup A_2\cup A_3\cup\{\alpha\}}=2v_{A_1}+2v_{A_2}+2v_{A_3}+2\cdot v_{A_4}$ by \cite{KM07}{Lemma 2.6}. Thus we get weight $-2$ if and only if $A_i=\{1,\ldots,j\}$ for $i=1,2,3$ or $4$ which is the case if and only if the $5$-valent vertex with $\alpha$ is adjacent to the $j+1$-valent vertex with $1,\ldots,j$.

The fifth case is analogous to the second case of Lemma \ref{lem:divintersect}. All neighbors have weight one. Therefore we get weight minus one if the $4$-valent vertex is adjacent to the $j+1$-valent vertex and $j$ is adjacent to the $j+1$-valent vertex. The claim follows.
\end{proof}

 Now  we intersect $\bbH^{\trop}_k(\mathbf{x})$ with the codimension $k$-skeleton of $\calM_{0,n}$. Let $K$ denote a cone of the codimension $k$-skeleton. It corresponds to a combinatorial type of a tree $\Gamma$ with $r-k$ vertices $V_1,\ldots,V_{r-k}$ of valence $\val(V_i)=k_i$ with $\sum (k_i-3)=k$. 
Remark \ref{rem:coneintersect} tells us how to pick functions $\varphi_1,\ldots,\varphi_k$ that cut out $K$ with weight one. Thus we want to compute 
$\bbH^{\trop}_k(\mathbf{x})\cdot\varphi_1\cdot\ldots\cdot\varphi_k $ locally around $K$. We have
\begin{align*}
& \bbH^{\trop}_k(\mathbf{x})\cdot\varphi_1\cdot\ldots\cdot\varphi_k = \\
& \ft_{\ast}\Big(\Psi_{n+1}\cdot \prod_{i=n+2}^{n+r-k} \big(\Psi_i \cdot \ev_i^{\ast}(p_i) \big) \Big) \cdot\varphi_1\cdot\ldots\cdot\varphi_k =\\
& \ft_{\ast}\Big(  \Psi_{n+1}\cdot \prod_{i=n+2}^{n+r-k} \big(\Psi_i \cdot \ev_i^{\ast}(p_i) \big) \cdot\ft^\ast(\varphi_1)\cdot\ldots\cdot\ft^\ast(\varphi_k)    \Big) =\\
& \ft_{\ast}\Big(  \prod_{i=n+1}^{n+r-k}\Psi_{i} \cdot\ft^\ast(\varphi_1)\cdot\ldots\cdot\ft^\ast(\varphi_k)  \cdot \prod_{i=n+2}^{n+r-k}  \ev_i^{\ast}(p_i)  \Big)
\end{align*}
where the second equality holds by the projection formula \cite{AR07}*{Proposition 4.8}.

To get a nonzero intersection of $\prod_{i=n+1}^{n+r-k}\Psi_{i} \cdot\ft^\ast(\varphi_1)\cdot\ldots\cdot\ft^\ast(\varphi_k) $ with the cycle $ \prod_{i=n+2}^{n+r-k}  \ev_i^{\ast}(p_i) $, the ends $n+1,\ldots, n+r-k$ must be adjacent to different vertices. The type $\Gamma$ corresponding to $K$ has $r-k$ vertices, and so we can attach one new end to each vertex of $\Gamma$. There are $m(\Gamma)$ ways to do this, where we use the notation from Definition \ref{def:weights} (we do not need to pick moving vertices here). Hence $\bbH^{\trop}_k(\mathbf{x})$ intersects $K$ in $m(\Gamma)$ points with a nonzero weight.
Each such point is the push-forward of an intersection point of  $ \prod_{i=n+2}^{n+r-k}  \ev_i^{\ast}(p_i) $ with a cone $\tilde{K}$ of $\prod_{i=n+1}^{n+r-k}\Psi_{i}$ where all ends are adjacent to different vertices. To compute the weight of $\tilde{K}$ in $\prod_{i=n+1}^{n+r-k}\Psi_{i} \cdot\ft^\ast(\varphi_1)\cdot\ldots\cdot\ft^\ast(\varphi_k) $, we use a generalization of Lemma \ref{lem:pullbackintersect} analogous to Remark \ref{rem:coneintersect}: in $\tilde{K}$, we have $r-k$ vertices $V_i$ of valence $k_i+1$ each of which is adjacent to an end with a Psi-class condition. We thus get  weight $\prod (k_i-2)$. From Lemma \ref{lem-evproduct}, the further intersection with $ \prod_{i=n+2}^{n+r-k}  \ev_i^{\ast}(p_i) $ yields a factor equal to the product of weights of all bounded edges.
We have thus proved the following statement that again illustrates the analogy between classical and tropical Hurwitz loci (compare with Lemma \ref{coeff}):

\begin{prop}
 Let $K$ be a cone of $\calM_{0,N}$ of codimension $k$, corresponding to the type $\Gamma$. 
Then the intersection $\bbH^{\trop}_k(\mathbf{x})\cdot K$ consists of $m(\Gamma)$ points, each with weight $\prod_v (\val(v)-2)\cdot \varphi(\Gamma)$
where the product goes over all vertices $v$ of $\Gamma$.
\end{prop}

\section{Wall Crossings}

We have seen that Hurwitz cycles are polynomials in each chamber $\fc$. In this section we investigate wall-crossings, i.e. how the cycles change from chamber to chamber. 
\begin{defn}
Let $I\subseteq \{1, \ldots, n\}$ and consider the wall $W_I=\{\sum_{i\in I}x_i =0\} $. Let $\fc^+$  and $\fc^-$ be two adjacent chambers: $\sum_{i\in I}x_i >0$ in $\fc^+$,  $\sum_{i\in I}x_i < 0$ in $\fc^-$ and for every $J\not=I \subseteq  \{1, \ldots, n\}$ the sign of $ \sum_{i\in J}x_i$ is the same in both chambers. Let $\bbH_k^+(\mathbf{x})$ (resp. $\bbH_k^-(\mathbf{x})$) be the polynomial class giving the Hurwitz cycle in $\fc^+$(resp. $\fc^-$). By {wall crossing formula} at the wall $I$ we mean the formal difference of cycles:
\begin{equation}
WC_{I,k}(\mathbf{x}):= \bbH_k^+(\mathbf{x}) -\bbH_k^-(\mathbf{x}) \in Z_k(\mon).
\end{equation}

\end{defn}
Naturally the difference of two polynomial cycles is a polynomial cycle: the upshot is that such a cycle can be expressed inductively in terms of Hurwitz cycles.

\begin{note} Let $\mathbf{x}$ and $\mathbf{y}$ be two tuples of integers such that $\sum{x_i}=-\sum{y_j}=\epsilon\not=0$. Consider the Hurwitz cycles $\bbH_{k_1}(\mathbf{x},-\epsilon)\in \overline{M}_{0,n_1+1}$ and $\bbH_{k_2}(\mathbf{y},\epsilon)\in \overline{M}_{0,n_2+1}$ and the gluing morphism 
$\gl:\overline{M}_{0,n_1+1}\times\overline{M}_{0,n_2+1}\to \overline{M}_{0,n_1+n_2}$. We denote:
$$
\bbH_{k_1}(\mathbf{x},-\epsilon)\boxtimes\bbH_{k_2}(\mathbf{y},\epsilon):= \gl_\ast\left( \bbH_{k_1}(\mathbf{x},-\epsilon)\times\bbH_{k_2}(\mathbf{y},\epsilon)\right) \in Z_{k_1+k_2}( \overline{M}_{0,n_1+n_2}).
$$
\end{note}

With this notation in place we are ready to state the wall crossing formulas.

\begin{thm}[Classical Wall Crossing]
\label{thm:wc}
Let $I\subseteq \{1, \ldots, n\}$ and consider the wall $W_I=\{\epsilon:=\sum_{i\in I}x_i =0\} $. Then:\begin{equation}
\label{eq:wc}
WC_{I,k}(\mathbf{x})= \epsilon \sum_{j=\max\{0,1+k-r_2\}}^{\min\{k,r_1-1\}}  {{r-k}\choose{|I|-1-j}}\bbH_j(\mathbf{x}_I,-\epsilon)\boxtimes\bbH_{k-j}(\mathbf{x}_{I^c},\epsilon)
\end{equation}
\end{thm}
\begin{proof}
The proof of Theorem \ref{thm:wc} is parallel to \cite{CJM10}*{Theorem 6.10}.
We first remark that the bounds of the summation are simply  recording the fact that $j$ (resp. $k-j$) must be less than or equal that the dimension of $\overline{M}_0^{\sim}(\mathbf{x}_I,-\epsilon)$ (resp. $\overline{M}_0^{\sim}(\mathbf{x}_{I^c},\epsilon)$). One may make the summation simply from $0$ to $k$ by noting that the Hurwitz loci of dimension greater than the corresponding moduli spaces of maps are empty.

Hurwitz cycles are completely described by the tropical dual graphs of the boundary strata in the moduli spaces of maps. In order for a tropical dual graph to contribute to the wall crossing, it must have an edge with weight equal to the equation of the wall. For a given tropical dual graph $\Gamma$, if such an edge exists, then it is unique and we call it the special edge. 

Cutting the special edge separates the graph into two subtrees  $\Gamma_I$ and $\Gamma_{I^c}$. 
The ends  of $\Gamma_I$ (resp. $\Gamma_{I^c}$) are labelled by $x_i\in I$ and $-\epsilon$ (resp. $x_i\in I^c$ and $\epsilon$). We note immediately that $\Gamma_I,\Gamma_{I^c}$ are a pair of graphs identifying a boundary stratum appearing in the product of  Hurwitz cycles on the right hand side of formula \eqref{eq:wc}. We make this connection more precise in order to extract quantitative information.

For $0\leq j\leq k$, let
$
R_j=\left\{ \left(\Gamma_1, \Gamma_2, \mathfrak{m} \right)\right\},
$
where:
\begin{itemize}
\item $\Gamma_1$ is the tropical dual graph of a stratum in  $\tilde{\bbH}_j(\mathbf{x}_I,-\epsilon)$ (pushing forward non-trivially to $\mon$).
\item $\Gamma_2$ is the tropical dual graph of a stratum in  $\tilde{\bbH}_{k-j}(\mathbf{x}_{I^c},\epsilon)$ (pushing forward non-trivially to $\mon$).
\item $\mathfrak{m}$ is a total ordering of the vertices of $\Gamma_1 \cup \Gamma_2$, compatible with the total ordering of the vertices of  $\Gamma_1$ and $\Gamma_2$.
\end{itemize}

Cutting the special edge gives a function $\Cut$ from the set of graphs contributing to the wall crossing formula to the union $R= \cup_{j=0}^k R_j$. We claim that $\Cut$ is a bijection, and it will come as little  surprise that the inverse function  $\Glue$ consists in gluing  the two graphs along the special edge labelled $\pm \epsilon$. The total ordering $\mathfrak{m}$ is precisely the information needed to make such gluing well defined. We note in particular that $\mathfrak{m}$ determines in which direction the special edge is pointing once it is glued.

Given a graph  $\Gamma$ contributing to the wall crossing, we note that the multiplicity it contributes to the wall crossing formulas is $\epsilon$ times the product of all weights of all non-special internal edges (this is obvious if $\Gamma$ comes from $\bbH_k^+(\mathbf{x})$. If  $\Gamma$ comes from $\bbH_k^-(\mathbf{x})$, then the weight of the special edge is $-\epsilon$, and there is another minus sign coming from the wall crossing formula). On the other hand the pair of graphs $\Gamma_1$ and $\Gamma_2$ in $\Cut(\Gamma)$ have multiplicity in $\bbH_j(\mathbf{x}_I,-\epsilon)\boxtimes\bbH_{k-j}(\mathbf{x}_{I^c},\epsilon)$ equal to the product of all non-special internal edges of $\Gamma$. Therefore $\Cut$ is a bijection that preserves the multiplicities on both sides of formula \eqref{eq:wc}.

 The proof is then concluded by remarking that if $\Gamma_1, \Gamma_2$ appear in $R_j$, there are ${{r-k}\choose{|I|-1-j}}$ possible ways of giving a total ordering of the vertices of $\Gamma_1 \cup \Gamma_2$, compatible with the total ordering of the vertices of  $\Gamma_1$ and $\Gamma_2$. 
\end{proof}

The wall crossing formula on the tropical side is similar: the only apparent difference is the lack of the multiplicative factor $\epsilon$, reflecting the fact that polynomiality does not appear in the generic representative of a tropical Hurwitz cycle.
Differently from the classical side however, it is not only the weights that depend on $\mathbf{x}$ but the cycles themselves, making even the statement of a wall crossing formula more subtle.

Fix a wall $W_I$ and two adjacent chambers $\fc^+$  and $\fc^-$ with $\epsilon:=\sum_{i\in I}x_i >0$ in $\fc^+$,  $\epsilon < 0$ in $\fc^-$. We denote by $\bbH^{\trop,+}_k(\mathbf{x})$ resp.\ $\bbH_k^{\trop,-}(\mathbf{x})$ the Hurwitz cycles in the two chambers.
We would like to evaluate  both $\bbH^{\trop,+}_k(\mathbf{x})$  and $\bbH_k^{\trop,-}(\mathbf{x})$
 at $\mathbf{x}\in \fc^+$
 and then consider the difference. However the fact that $\epsilon$ changes sign when crossing the wall requires some interpretation, since the edge lengths of tropical covers
are required to be positive. 
Consider a point in $\bbH^{\trop,-}_k(\mathbf{x})$ corresponding to a tropical cover with an edge of weight $-\epsilon$ connecting two vertices that are mapped to two points $p<q$ in $\TP$. The length of the special edge of this tropical cover then equals $\frac{q-p}{-\epsilon}$, which is positive in $\fc^-$, but negative in $\fc^+$. In the latter case, $\bbH^{\trop,-}_k(\mathbf{x})$ does not live in  $\calM_{0,n}$ but in the real vector space $\mathbb{R}\calM_{0,n}$ surrounding it.
To define $\bbH^{\trop,-}_k(\mathbf{x})\subset \calM_{0,n}$ for $\mathbf{x}\in \fc^+$, we consider  the linear map $c_\epsilon$ that maps $v_I$ to itself for all $I\neq \epsilon$ and $v_\epsilon$ to $-v_\epsilon$, and redefine $\bbH^{\trop,-}_k(\mathbf{x})$ to be the image of $\bbH^{\trop,-}_k(\mathbf{x})\subset \mathbb{R}\calM_{0,n}$  under $c_\epsilon$.

Also the weights of the cycle $\bbH^{\trop,-}_k(\mathbf{x})$ require some interpretation: in Lemma \ref{lem:cellsweights} we have seen that the weight of a cell equals the gcd of products of weights of trees where we shrink edges in such a way that all moving vertices merge with a non-moving vertex. When computing $\bbH_k^{\trop,-}(\mathbf{x})$, we use the convention to use the negative gcd whenever $-\epsilon$ appears in a product.

\begin{defn}
With the notation from the above paragraph, we now define the \emph{tropical Wall Crossing} as
$$WC^{\trop}_{I,k}(\mathbf{x}):= \bbH^{\trop,+}_k(\mathbf{x}) - \bbH^{\trop,-}_k(\mathbf{x}).$$
\end{defn}

\begin{example}\label{ex:wcm05}
 In Example \ref{ex:m05hurwitzcurve}, we  consider the Hurwitz curve for  $\mathbf{x}=(x_1,
\ldots,x_5)$ in $\fc^+$ defined by $x_1, x_2>0$, $x_3,x_4,x_5<0$, $x_1>|x_i+x_j|$ for  $i,j \in\{3,4,5\}$ and $x_2<-x_i$ for  $i=3,4,5$ in $\calM_{0,5}$. We now cross the wall $\epsilon:=x_1+x_4+x_5=0$. 

$\bbH^{\trop,+}_1(\mathbf{x}) $ and $ \bbH^{\trop,-}_1(\mathbf{x})$ only differ in cones whose corresponding type contains an edge with weight $\pm\epsilon$, i.e.\ in any cone containing the vector $v_{23}$. There are three such cones. For two of these cones, the Hurwitz curve $\bbH^{\trop,+}_1(\mathbf{x}) $ looks as depicted in Figure \ref{fig:m052} and for the third cone as in Figure \ref{fig:m054}. In $ \bbH^{\trop,-}_1(\mathbf{x})$, any edge with weight $-\epsilon$ changes direction. In the cone spanned by $v_{23}$ and $v_{14}$, if we choose the vertex adjacent to end $5$ to be the moving vertex, we now have two ways to order the remaining vertices, both corresponding to linear ends. Thus one linear edge is replaced by two linear ends in this cone. The same is true for the cone spanned by $v_{23}$ and $v_{15}$ by symmetry. In the cone spanned by $v_{23}$ and $v_{45}$ contrarily, if the moving vertex is adjacent to end $1$, we have only one way to order the remaining vertices corresponding to a linear edge instead of the two linear ends (see Figures \ref{fig:m053} and \ref{fig:m054}). 
Note also that the constant ends of direction $v_{23}$ do not have a factor of $\epsilon$ in their weight, thus they appear with the same sign both in $\bbH^{\trop,+}(_1\mathbf{x}) $ and $ \bbH^{\trop,-}_k(\mathbf{x})$ and cancel in the difference. The other constant ends have weight $\epsilon$ in $\bbH^{\trop,+}_1(\mathbf{x}) $ but weight $-\epsilon$ in $ \bbH^{\trop,-}_k(\mathbf{x})$, so they do not cancel but add up to a constant end with weight $2\epsilon$. Figure \ref{fig:wc} depicts the tropical wall crossing curve for these two chambers. Blue edges are edges of $\bbH^{\trop,+}_1(\mathbf{x}) $, red edges are edges of $ \bbH^{\trop,-}_k(\mathbf{x})$. The green constant ends appear in both and add up to the weight $2\epsilon$. The picture only shows the three cones of $\calM_{0,5}$ in which the difference $\bbH^{\trop,+}_1(\mathbf{x}) - \bbH^{\trop,-}_k(\mathbf{x})$ is nonzero.

\begin{figure}[tb]
\input{wc.pstex_t}
\caption{The wall crossing curve in $\calM_{0,5}$ for $\fc^+$ and $\fc^-$.}
\label{fig:wc}
\end{figure}

\end{example}

As in the classical world, we want to describe a tropical wall crossing in terms of cutting and regluing. 
 Any cell contributing to the wall crossing parameterizes graphs with a special edge that we may cut to obtain two subgraphs each of which is a tropical cover of the projective line.  Remembering the length of the edge that gets cut accounts for the fact that we want to mod out by translations only once. We make this process precise in the following paragraph.

\begin{construction}
 
Let $I\subseteq \{1, \ldots, n\}$ and consider the wall $W_I=\{\epsilon:=\sum_{i\in I}x_i =0\} $. 
Assume that $|I|=n_1$ and $|I^c|=n_2$ and denote $r_i=n_i-1$ for $i=1,2$.
Let $\Gamma$ be a tropical cover in a $k$-dimensional cell that contributes to the wall crossing, and hence it contains an edge $e$ with weight $\pm\epsilon$. Cutting $e$ we obtain two subgraphs $\Gamma_1$ and $\Gamma_2$ that are themselves tropical covers of the projective line. We assume that $\Gamma_1$ contains the ends in $I$ and $\Gamma_2$ contains the ends in $I^c$. Both $\Gamma_1$ and $\Gamma_2$ have an extra end that we denote by $E_1$ (resp.\ $E_2$). According to our conventions ends are oriented inward, and hence the balancing condition gives $E_1$ weight $-\epsilon$ and $E_2$ weight $\epsilon$. 
Assuming that $r_1-j$ fixed vertices are in $\Gamma_1$,  we can interpret $\Gamma_1$ as an element in $\calM_{0,r_1-j}(\TP,(\mathbf{x}_I,-\epsilon))$ and $\Gamma_2$ as an element in $\calM_{0,r_2-(k-j)}(\TP,(\mathbf{x}_{I^c},\epsilon))$ (adjusting the labeling of the ends). We wish to remember the length of the edge we cut and whether $\Gamma$ belonged to $\bbH^{\trop,+}_1(\mathbf{x}) $ or $\bbH^{\trop,-}_1(\mathbf{x}) $, so we define:
$
\Cut (\Gamma) \in \left(\calM_{0,r_1-j}(\TP,(\mathbf{x}_I,-\epsilon))\times \IR\right)\times \left(\calM_{0,r_2-(k-j)}(\TP,(\mathbf{x}_{I^c},\epsilon))\times \IR\right)
$
by 

\begin{equation} 
\label{eq:cut}
\Cut(\Gamma) =\begin{cases} \left((\Gamma_1, 0) , (\Gamma_2, l(e))\right) & \Gamma \in  \bbH^{\trop,+}_1(\mathbf{x}) \\  \left((\Gamma_1, 0) , (\Gamma_2, -l(e))\right) & \Gamma \in  \bbH^{\trop,-}_1(\mathbf{x}). \end{cases}
\end{equation}

\begin{rem} 
Our choice to have two $\IR$ coordinates and assigning one of them to be $0$ seems, and in fact is, somewhat arbitrary. However, it will be handy when comparing weights of the same cells appearing on opposite sides of the wall crossing formula.
\end{rem}
\begin{rem}
In order for  equation \eqref{eq:cut} to make sense we need to drop Convention \ref{conv} (see page \pageref{conv}) and remember tropical covers are equivalent up to translation. We therefore use one of the vertices in each of the subgraphs to fix a parameterization of $\TP$.
\end{rem}

We reverse this operation to glue two graphs in  $\left(\calM_{0,r_1-j}(\TP,(\mathbf{x}_I,-\epsilon))\times \IR\right) \times\left(\calM_{0,r_2-(k-j)}(\TP,(\mathbf{x}_{I^c},\epsilon))\times \IR\right)$. Denote by $l_i$ the $\IR$ coordinate function for the $i$-th factor in the product.
If $V_i$ is the interior vertex of $\Gamma_i$ adjacent to $E_i$, then  define $\ev_{E_i}: \calM_{0,r_1-j}(\TP,(\mathbf{x}_I,-\epsilon))\times \IR\to \IR$ by $$\ev_{E_i}(\Gamma_i, l_i)=\ev_{V_i}+(-1)^il_i\cdot \epsilon. $$

We can glue any two pieces in $(\ev_{E_1}-\ev_{E_2})^\ast(0)\cdot (\calM_{0,r_1-j}(\TP,(\mathbf{x}_I,-\epsilon))\times \IR) \times (\calM_{0,r_2-(k-j)}(\TP,(\mathbf{x}_{I^c},\epsilon))\times \IR) $ to one tropical cover in $\calM_{0,r-k}(\TP,\mathbf{x})$. To make this operation the inverse to $\Cut$ defined above, we further impose $l_1=0$. 


The above discussion shows that 
\begin{align*}\mathcal{G} := (\ev_{E_1}-\ev_{E_2})^\ast(0)\cdot l_1^\ast(0) \cdot  \Big( &(\calM_{0,r_1-j}(\TP,(\mathbf{x}_I,-\epsilon))\times \IR) \times \\  &  (\calM_{0,r_2-(k-j)}(\TP,(\mathbf{x}_{I^c},\epsilon))\times \IR) \Big)\end{align*}
 is in bijection to the set of covers in $\calM_{0,r-k}(\TP,\mathbf{x})$ contributing to the wall crossing 
(see \cite{GMO} for a similar gluing construction for moduli spaces). 

We now define the folding map $\fold:\mathcal{G} \rightarrow \calM_{0,r-k}(\TP,\mathbf{x})$ that maps $l_2$ to its absolute value.
The folding map is not globally a tropical morphism, it is only locally a morphism away from $\mathcal{G}\cdot l_2^\ast(0)$.
Consequently, while $\mathcal{G}$
is a tropical variety, its image under $\fold$ is not. Since the wall crossing curve is also not a tropical variety however, this is not disturbing.

To make the image of $\fold$ a weighted polyhedral complex, we give each cell the sum of the weights of its preimages (cells are subdivided by 
$\mathcal{G}\cdot l_2^\ast(0))$. This coincides with the weight of the push forward of $\fold$ locally where it is a morphism.
\end{construction}

We  now state the first version of the tropical wall crossing. 
 \begin{prop}[Tropical Wall Crossing, first version]
Let $I\subseteq \{1, \ldots, n\}$ and consider the wall $W_I=\{\epsilon:=\sum_{i\in I}x_i =0\} $. Then:

\begin{equation}\hspace{-1cm}
\begin{array}{cl}
WC^{\trop}_{I,k}(\mathbf{x}) &=  \ft_{\ast} \bigg(  \sum_{j=\max\{0,1+k-r_2\}}^{\min\{k,r_1-1\}} \;\;\;\;  \sum_{n+1 \leq i_1<\ldots<i_{r_1-j}\leq n+r-k}  
 \fold   \Big(  \\ & \\
 & \left.\left.  \hspace{-1cm}
\Big( \Psi_{n_1+2}\cdot \prod_{l=2}^{r_1-j} \big(\Psi_{l+n_1+1} \cdot \ev_{l+n_1+1}^{\ast}(p_{i_{l}}) \big)      
\cdot \calM_{0,r_1-j}(\TP,(\mathbf{x}_I,-\epsilon)) \Big) \times \IR  
\times   \right.\right.\\ & \\
& \left.\left.\hspace{-1cm}
\Big( \Psi_{n_2+2}\cdot \prod_{l=2}^{r_2-(k-j)} \big(\Psi_{l+n_2+1} \cdot \ev_{l+n_2+1}^{\ast}(p_{j_{l}}) \big)   \cdot \calM_{0,r_2-(k-j)}(\TP,(\mathbf{x}_{I^c},\epsilon)) \Big) \times \IR \times \right.\right.\\ & \\
&\hspace{-1cm}
 (\ev_{E_1}-\ev_{E_2})^\ast(0)\cdot l_1^\ast(0)
\Big)\bigg), 
\end{array}
\label{twc-one}
\end{equation}
where $j_1<\ldots<j_{r_2-(k-j)}$ are defined by $$\{n+1,\ldots,n+r-k\}=\{i_1,\ldots,i_{r_1-j}, j_1,\ldots, j_{r_2-(k-j)}\}.$$ 

\end{prop}

\begin{proof}
 We first show that graphs where the special edge separates two pieces $\Gamma_1$ and $\Gamma_2$ such that there is no fixed (i.e. $4$-valent) vertex in $\Gamma_1$ do not contribute to the wall crossing. 
 Analogously the same is true for types with no $4$-valent vertex in $\Gamma_2$. 
 This justifies the bounds of the first sum: $j \geq 1+k-r_2$ is equivalent to $r_2-(k-j)\geq 1$ which means that there is at least one $4$-valent vertex in $\Gamma_2$, the upper bound $j\leq r_1-1$ guarantees that there is at least one $4$-valent vertex in $\Gamma_1$. 

Assume $\Gamma$ is a type with no $4$-valent vertex in $\Gamma_1$. Then there are only moving vertices in $\Gamma_1$. The type $\Gamma$ appears both in $\bbH_k^{\trop,+}(\mathbf{x}) $ and in $\bbH_k^{\trop,-}(\mathbf{x}) $ : the special edge can point in any direction if none of the vertices past  the special edge  are free to move. The corresponding cells are identical set-theoretically.
 To compute their weight, we have to consider types $\Gamma'$ where all moving vertices merge with some fixed vertices. In any such type, the special edge must be shrunk since there are only moving vertices on one side of it. Hence the weight which equals the gcd of products of weights of possible $\Gamma'$ is identical for both cells, which thus cancel in the difference $\bbH_k^{\trop,+}(\mathbf{x}) - \bbH_k^{\trop,-}(\mathbf{x})$.

The above construction shows that  formula \eqref{twc-one} holds set-theoretically. 
It remains to show that also the weights agree. Since the cells of $\tilde{\bbH}^{\trop}_k(\mathbf{x})$ are in weight-preserving bijection with the cells of the push forward $\bbH^{\trop}_k(\mathbf{x}) \subset \calM_{0,n}$, we can show the equality of the weights for the cells before pushing forward with the forgetful map. Also, we can separate the contributions from the two sides of the wall, i.e. ``unfold'' on the cut-and-glue side of the equation.
Let $K$ be a cell 
contributing to the wall crossing
  corresponding to the combinatorial type $\Gamma$.
  Following the construction above we produce $\Gamma_1$ and $\Gamma_2$ which  identify a unique cell $K'$ on the right hand side of equation \eqref{twc-one}.  By Remark \ref{rem-intersect}  the weight of $K$ is the gcd of the absolute values of the maximal minors of the matrix of the evaluation maps evaluating the fixed vertices of $\Gamma$ in local coordinates. Analogously, the weight of the corresponding cut-and-glue cell $K'$ is the gcd of the absolute values of the maximal minors of the matrix of the evaluation maps evaluating the fixed vertices of $\Gamma_1$, the fixed vertices of $\Gamma_2$, of the map $\ev_{E_1}-\ev_{E_2}$
  and of the map $l_1$.
Local coordinates of the cell $K$ are given by the lengths of all bounded edges of $\Gamma$,
whereas for the cell $K'$ they are given by the bounded edges of $\Gamma$ except the special edge, plus the lengths $l_1$ and $l_2$. The map $l_1$ produces a row for the matrix with a one in the $l_1$-column and only zeros in the other columns. Any nonzero maximal minor must contain this column and equals the determinant of the submatrix where we erase the $l_1$-column and the row for the map $l_1$. Therefore we can erase this column and row without changing any minor in the matrix for $K'$. We show that the matrix for $K$ and the reduced matrix for $K'$ coincide up to row operations. Since row operations do not change the absolute value of any maximal minor, the claim follows. In the matrix for $K$, columns are the lengths in $\Gamma_1$, the length $l$ of the special edge and the lengths in $\Gamma_2$. As rows we have the evaluations of the contracted ends $n+2,\ldots,n+r-k$ which for each end $i$ are given by following the unique path in $\Gamma$ from the reference end $n+1$ (that we require to map to $0$) to $i$.
In the reduced matrix for $K'$, we have the analogous columns. Assume without restriction that $i_1=n+1$, i.e.\ the end mapping to zero is contained in $\Gamma_1$. We evaluate all other contracted ends of $\Gamma_1$ with respect to this reference end, i.e.\ we follow the path from this end to any other in $\Gamma_1$. This gives exactly the same rows for the evaluation of ends of $\Gamma_1$. In the matrix of $K$, we have a row evaluating the end that becomes the reference end in $\Gamma_2$ (that we fix to be at the point $p_{j_1}$). This row is exactly equal to the row evaluating $\ev_{E_1}-\ev_{E_2}$ --- the unique path in $\Gamma$ connecting the two reference ends. The remaining rows evaluate ends in $\Gamma_2$. For an end $i$ in $\Gamma_2$, we evaluate it in $K$ by following the path from the reference end in $\Gamma_1$ to $i$. We evaluate it in $K'$ by following the path from the reference end in $\Gamma_2$ to $i$. These two rows just differ by the path connecting the two reference ends which by the above is also a row of both matrices. Hence we can perform row operations
 to produce the matrix for $K$ from the matrix for $K'$.
\end{proof}

\begin{example}
We establish the wall crossing formula \eqref{twc-one} for the wall crossing curve in $\calM_{0,5}$ from Example \ref{ex:wcm05}. 
The wall $\epsilon=x_1+x_4+x_5=0$ separates any cover $\Gamma$ containing an edge $\epsilon$ into $\Gamma_1$ and $\Gamma_2$, where $\Gamma_1$ has the ends $x_1,x_4,x_5$ and $-\epsilon$ and $\Gamma_2$ has the ends $x_2$, $x_3$ and $\epsilon$. We have two contracted marked ends. Since both pieces must contain at least one to have at least one $4$-valent vertex, both contain exactly one. 
We first assume that the end $6$  (required to map to zero) is in $\Gamma_1$. For $\Gamma_2$, there is exactly one type, where all four ends (including the contracted end $7$) are incident to one $4$-valent vertex. We have  a copy of $\IR$ parametrizing all possibilities with coordinate $l_2$.

For $\Gamma_1$, there are six possible types, since there are three types in $\calM_{0,4}$ each of which has two $3$-valent vertices, and we can decide which of the two vertices should be equipped with the contracted end $6$ and become $4$-valent. 
We neglect the length of the extra end $l_1$ that is required to be zero.
We  have a ray for each of the six types, the local coordinate is the length of the bounded edge that we denote by $l$. 
Figure \ref{fig:6types} shows the six types together with the part to which they are glued.

\begin{figure}[tb]
\input{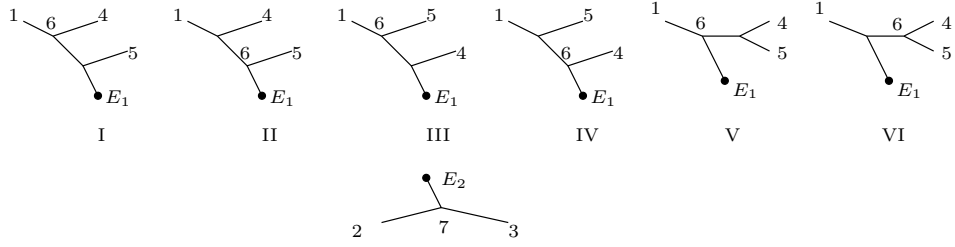}
\caption{The possible types for $\Gamma_1$ and $\Gamma_2$ after cutting, assuming $6$ is in $\Gamma_1$.}
\label{fig:6types}
\end{figure}

For each type, we  write down the gluing equation $\ev_{E_1}-\ev_{E_2}=0$ in local coordinates. Remember we require $6$ to be mapped to zero and $7$ to one.
The equations are:

\begin{align*}
 \mathrm{I}:\qquad &  0+l (x_1+x_4)+\epsilon l_2=1, \\
 \mathrm{II, IV,V}:\qquad &  0+l_2\epsilon=1,\\
 \mathrm{III}:\qquad &0+l (x_1+x_5)+\epsilon l_2=1, \\
 \mathrm{VI}:\qquad & 0-l (-x_4-x_5)+l_2\epsilon=1.
\end{align*}

These gluing equations cut out the tropical variety depicted in Figure \ref{fig:gluing} from $\Psi_6\cdot \calM_{0,1}(\TP,x_1,x_4,x_5,-\epsilon)\times \Psi_7\cdot \calM_{0,1}(\TP,x_2,x_3,\epsilon)\times \IR$ which consists of six rays 
We also draw the image under the folding map. The rays that are mapped from the orthant where $l_2$ is negative via the folding map are drawn with dotted lines.

\begin{figure}[tb]
\input{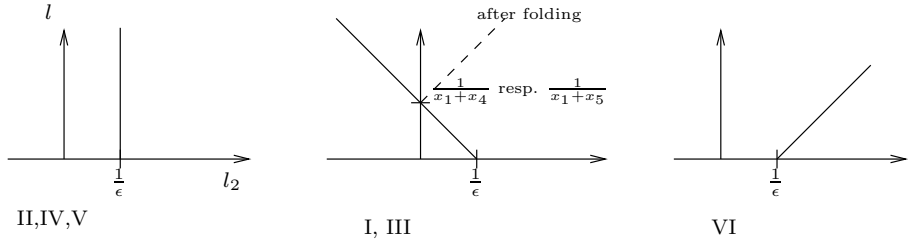}
\caption{The variety obtained from gluing and its image under folding for the first summand.}
\label{fig:gluing}
\end{figure}

Analogously, if end $7$ is in $\Gamma_1$ and $6$ is in $\Gamma_2$, we have the types depicted in Figure \ref{fig:6types} but with $6$ and $7$ exchanged.
The gluing equations in local coordinates are as above, but with the role of $0$ and $1$ exchanged.
Figure \ref{fig:gluing2} shows the tropical variety cut out by the gluing equation from $\Psi_7\cdot \calM_{0,1}(\TP,x_1,x_4,x_5,-\epsilon)\times \Psi_6\cdot \calM_{0,1}(\TP,x_2,x_3,\epsilon)\times \IR$ and its image under folding. As before, rays that appear in the positive orthant after folding are drawn as dotted lines.

\begin{figure}[tb]
\input{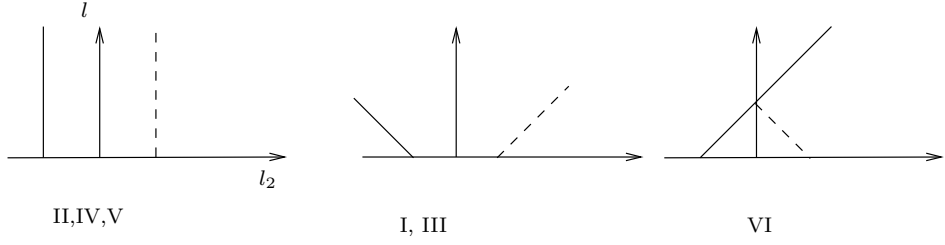}
\caption{The variety obtained from gluing and its image under folding for the second summand.}
\label{fig:gluing2}
\end{figure}

Under the forgetful map, the four orthants labeled I and II go to the cone spanned by $v_{23}$ and $v_{14}$, the four orthants labeled III and IV go to the cone spanned by $v_{23}$ and $v_{15}$ and the remaining four to the cone spanned by $v_{23}$ and $v_{45}$. We get the picture  seen in  Example \ref{ex:wcm05}.

\end{example}

We now group summands in the cut-and-glue side of the equation to give a nicer interpretation of formula \eqref{twc-one}.
If we denote by $\tilde{\bbH}^{\trop}_k(\mathbf{x})$ the preimage of $\bbH^{\trop}_k(\mathbf{x})$ under the forgetful map, then as in Remark \ref{rem:equivalent} the rational equivalence class of the loci 
\begin{align*}& \Big( \Psi_{n_1+2}\cdot \prod_{l=2}^{r_1-j} \big(\Psi_{l+n_1+1} \cdot \ev_{l+n_1+1}^{\ast}(p_{i_{l}}) \big)      
\cdot \calM_{0,r_1-j}(\TP,(\mathbf{x}_I,-\epsilon)) \Big) \times \IR  \\&\sim \tilde{\bbH}^{\trop}_j(\mathbf{x}_I,-\epsilon) \times \IR\end{align*}
 does not depend on the exact choice of $i_1<\ldots<i_{r_1-j}$, and analogously for the other loci in the sum. Up to rational equivalence we can thus group all summands of the second sum into one just counting how many such summands there are: $\binom{r-k}{r_1-j}=\binom{r-k}{|I|-1-j}$. Then \eqref{twc-one} becomes
\begin{align*}
 &WC^{\trop}_{I,k}(\mathbf{x})\sim\\ & \ft_{\ast} \Bigg(  \sum_{j=\max\{0,1+k-r_2\}}^{\min\{k,r_1-1\}}  
\binom{r-k}{r_1-j}  \fold   \bigg( \Big(\tilde{\bbH}^{\trop}_j(\mathbf{x}_I,-\epsilon) \times \IR \times \tilde{\bbH}^{\trop}_{k-j}(\mathbf{x}_{I^c},\epsilon)\times \IR \Big)
\\&
\times 
 (\ev_{E_1}-\ev_{E_2})^\ast(0)\cdot l_1^\ast(0)\bigg)\Bigg).
\end{align*}

If we furthermore denote
\begin{align*}
 &\bbH^{\trop}_j(\mathbf{x}_I,-\epsilon)\boxtimes\bbH^{\trop}_{k-j}(\mathbf{x}_{I^c},\epsilon) :=\\&
\fold   \left( \Big(\tilde{\bbH}^{\trop}_j(\mathbf{x}_I,-\epsilon) \times \IR \times \tilde{\bbH}^{\trop}_{k-j}(\mathbf{x}_{I^c},\epsilon)\times \IR \Big) (\ev_{E_1}-\ev_{E_2})^\ast(0)\cdot l_1^\ast(0)\right)
\end{align*}
we obtain a formula for the tropical wall crossing that looks very similar to the classical wall crossing formula  \eqref{eq:wc}.

\begin{cor}[Tropical Wall Crossing, reloaded]
\label{cor:twc}
Let $I\subseteq \{1, \ldots, n\}$ and consider the wall $W_I=\{\epsilon:=\sum_{i\in I}x_i =0\} $. Then:
\begin{align*}
 &WC^{\trop}_{I,k}(\mathbf{x})\sim\\ & \ft_{\ast} \left( \sum_{j=\max\{0,1+k-r_2\}}^{\min\{k,r_1-1\}}  \binom{r-k}{|I|-1-j}          \bbH^{\trop}_j(\mathbf{x}_I,-\epsilon)\boxtimes\bbH^{\trop}_{k-j}(\mathbf{x}_{I^c},\epsilon) \right).
\end{align*}

\end{cor}



\subsection{Hurwitz Curves: Take three}

A few features of Example \ref{ex:wcm05} can be generalized to the general one-dimensional wall crossing  case immediatly. Constant ends of direction $v_I$ cancel in the wall crossing. Constant ends of other directions can appear on both sides of the wall but will show up in the wall crossing with a positive sign. Let $\Gamma^+$ be a graph in $\tilde{\bbH}^{\trop,+}_k(\mathbf{x})$ with a special edge and $V$ a choice of a moving vertex corresponding to a linear end. Then we have $m_V(\Gamma^+)$ linear ends in $\tilde{\bbH}^{\trop,+}_k(\mathbf{x})$, and $m_V(\Gamma^-)$ linear edges in $\tilde{\bbH}^{\trop,-}_k(\mathbf{x})$ (where $\Gamma^-$ denotes the graph with the special edge reversed). All such edges come with the same weight in the wall crossing.  Analogously, $m_V(\Gamma^+)$ quadratic edges are replaced by $m_V(\Gamma^-)$ quadratic edges when crossing the wall, also coming with the same weight.

\bibliographystyle{amsalpha}
\bibliography{biblio}

\end{document}